\documentclass{amsart}
\usepackage[utf8]{inputenc}
\usepackage{amsmath, amssymb}
\usepackage{amsthm}
\usepackage[all]{xy}
\usepackage{color}
\definecolor{darkblue}{rgb}{0,0,0.6}
\usepackage[ocgcolorlinks,colorlinks=true, citecolor=darkblue, filecolor=darkblue,
linkcolor=darkblue, urlcolor=darkblue]{hyperref}
\usepackage[nameinlink,capitalise,noabbrev]{cleveref}
\usepackage{bbm}
\usepackage{colonequals}  
\usepackage{enumerate}
\usepackage{verbatim}  

\makeatletter
\@namedef{subjclassname@2010}{%
  \textup{2010} Mathematics Subject Classification}
\makeatother
\title{The Quasi-hyperbolicity Constant of a Metric Space}
\author[Dragomir]{George Dragomir}
\address{George Dragomir, Department of Mathematics, Columbia University, New York, NY, USA 10027}
\email{dragomir@math.columbia.edu}

 \author[Nicas]{Andrew Nicas}
\address{Department of Mathematics and Statistics, McMaster University, Hamilton, Ontario, Canada L8S 4K1}
\email{nicas@mcmaster.ca}
\thanks{Andrew Nicas was partially supported by a grant from the Natural Sciences and Engineering Research Council of Canada}

\newcounter{commentcounter}


\date{\today}

\newcommand{\IR}{\mathbb{R}}
\newcommand{\R}{\mathbb{R}}

\DeclareMathOperator{\CAT}{CAT}

\numberwithin{equation}{section}
\newtheorem{thm}[equation]{Theorem}

\newtheorem{conj}[equation]{Conjecture}
\newtheorem{prop}[equation]{Proposition}
\newtheorem{cor}[equation]{Corollary}
\newtheorem{lemma}[equation]{Lemma}

\newtheorem{cor+}{Corollary}
\newtheorem{prop+}[cor+]{Proposition}
\theoremstyle{definition}

\newtheorem{example}[equation]{Example}
\newtheorem{defi}[equation]{Definition}
\newtheorem{question}[equation]{Question}
\newtheorem*{definition}{Definition}
\newtheorem*{thmm}{Theorem}
\newtheorem*{corollary}{Corollary}
\newtheorem{rem}[equation]{Remark}

\calclayout

\begin{document}

\begin{abstract}
 We introduce the {\it quasi-hyperbolicity constant} of a metric space,
a rough isometry invariant that measures how a metric space  deviates from being Gromov hyperbolic.
Gromov hyperbolicity, and also the lack thereof, has attracted considerable interest in the theory of networks.
The quasi-hyperbolicity constant for an unbounded space lies in the closed interval $[1,2]$.
It is equal to one for an unbounded  Gromov hyperbolic space.
For a  CAT$(0)$-space, it is bounded from above by $\sqrt{2}$.
The quasi-hyperbolicity constant of a Banach space that is at least two dimensional is bounded from below by $\sqrt{2}$,
and for a non-trivial $L_p$-space it is exactly $\max\{2^{1/p},2^{1-1/p}\}$.
If $0 < \alpha < 1$ then the quasi-hyperbolicity constant of the $\alpha$-snowflake of any metric space is bounded from above by $2^\alpha$.
We give an exact calculation in the case of the $\alpha$-snowflake of the Euclidean real line.
\end{abstract}

\maketitle

\baselineskip 15pt

\section{Introduction}

        Gromov hyperbolic spaces were introduced by Gromov in his seminal paper \cite{Gro:87Hyp:aa}
to study infinite groups as geometric objects.
For a metric space  $(X,d)$,
we use the abbreviated notation 
$xy=d(x,y)$ where convenient.
Recall that for three points  $x,y,w$ in a metric space $(X,d)$, the {\it Gromov product} of $x$ and $y$ with respect to $w$ is defined~as
\begin{equation*}
(x\mid y)_w = \tfrac{1}{2}\left( xw+yw - xy \right).
\end{equation*} 
Given a non-negative constant $\delta$, the metric space $(X,d)$ is said to be {\it $\delta$-hyperbolic} if
\[
(x\mid y)_w\ge\min\left\{(x\mid z)_w, (y\mid z)_w\right\}-\delta
\]
for all $x,y,z,w\in X$.
A metric space $(X,d)$ is said to be {\it Gromov hyperbolic} if it is  
$\delta$-hyperbolic for some $\delta$.
Any $\IR$-tree is $0$-hyperbolic. 
Another well-known example is the hyperbolic plane, which is $\log(2)$-hyperbolic,  \cite[Corollary 5.4]{Nica-Spakula}.
Euclidean spaces of dimension greater than one are not Gromov hyperbolic. 
While Gromov hyperbolicity is a quasi-isometry invariant for {\it intrinsic} metric spaces \cite[Theorems 3.18 and 3.20]{Vai:05Gro:aa},
quasi-isometry invariance can fail for non-intrinsic spaces, see \cite[Remark~3.19]{Vai:05Gro:aa} and also our examples in \S\ref{sec:basic_examples}.
In particular, a metric space that quasi-isometrically embeds into a Gromov hyperbolic space need not be Gromov hyperbolic.

A metric space $(X,d)$ is $\delta$-hyperbolic if and only if the {\it the four-point inequality} holds, that is,
for all $x,y,z,w\in X$,
 \begin{equation*} 
xy + zw \le\max\{xz+yw, \, yz+xw\}+2\delta,
\end{equation*}
see  \cite[(2.12)]{Vai:05Gro:aa}.   

We generalize the four-point inequality as follows.
Let $(X,d)$ be a metric space.
Let $\mu, \delta \geq 0$.
We say that a metric space $(X,d)$ satisfies the \textit{$(\mu,\delta)$-four-point inequality} if for all $x,y,z,w \in X$,
\[
xy+zw\le \mu\max\{xz+yw, \, xw+yz\}+2\delta.
\]
In particular,  $(X,d)$ is $\delta$-hyperbolic if and only if it satisfies the $(1,\delta)$-four-point inequality.

We introduce the following numerical constants associated to a metric space.

\begin{definition}{\rm (Quasi-hyperbolicity and quadrilateral constants)} 
Let $(X,d)$ be a metric space.
\begin{itemize}
\item[(i)]  The {\it quasi-hyperbolicity constant} of $(X,d)$ is the number
\begin{align*}\label{eq:qhc}
C(X,d)   = \inf\{\mu ~|~  \mbox{ there exists $\delta \geq 0$}& \mbox{ such that $(X,d)$ satisfies} \\
& \mbox{the $(\mu,\delta)$-four-point inequality}\}. 
\end{align*}

\item[(ii)] The {\it quadrilateral constant} of $(X,d)$ is the number  
\begin{equation*}\label{eq:qhc0}
C_0(X,d) = \inf\{\mu ~|~  \text{$(X,d)$ satisfies the $(\mu,0)$-four-point inequality}\}. 
\end{equation*}

\end{itemize}
\end{definition}


Our motivation for the introduction of the quasi-hyperbolicity constant
and the closely related quadrilateral constant
originates from the theory of networks.
Many complex systems can be modeled by finite metric spaces
and the geometry of these spaces is related to their structure and function.
For example, the concept of a $\delta$-hyperbolic metric space has been effectively applied to network security, \cite{JL2004, MR2837770},
and to biological and social networks, \cite{PhysRevE.89.032811}.
Note that when considering a finite metric space, Gromov's constant $\delta$ should be taken to be appropriately smaller than the diameter of the space
as otherwise the four-point inequality would be trivial.
However, not every interesting network is $\delta$-hyperbolic, see \cite{MR2902244}. 
Unbounded metric spaces, which are the principal focus of this paper,
are still useful in the context of finite metric spaces as there is utility in embedding a finite metric space with controlled distortion and roughness into large metric spaces with well-understood properties,
in particular with known quasi-hyperbolicity and quadrilateral constants.
The quasi-hyperbolicity constant of an unbounded metric space measures its deviation from being Gromov hyperbolic.

Some basic properties of the quasi-hyperbolicity  and quadrilateral constants of a metric space $(X,d)$ are readily derived,
for example:
\begin{itemize}
\item $C(X,d) \leq C_0(X,d) \leq 2$,
\item if $(X,d)$ is bounded then $C(X,d) =0$, otherwise $C(X,d) \geq 1$,
\item if $(X,d)$ has at least two points then it is $0$-hyperbolic if and only if $C_0(X,d) =1$,
\item if $(X,d)$ is Gromov hyperbolic and unbounded then $C(X,d) =1$.
\end{itemize}
Proofs of these and more properties are given in \S~\ref{sec:elementary}.

In the absence of additional hypotheses,
it is not true that $C(X,d) =1$ implies $(X,d)$ is Gromov hyperbolic.
For example, given $0< \alpha < 1$,
consider the graph, $Y_\alpha$,
of $y = x^\alpha$, $x \geq 0$, as a subspace of the Euclidean plane, $(\IR^2,d_E)$.
We show  $C(Y_\alpha, d_E) =1$, Proposition~\ref{prop:quasiconstantone}, 
however
$Y_\alpha$ is {\it not}
Gromov hyperbolic
if and only if $1/2 < \alpha < 1$,
Propositions~\ref{prop:lessthanorequaltoonehalf}  and ~\ref{prop:notgromovhyperbolic}.
Nevertheless, 
if $(X,d)$ is a proper $\CAT(0)$-space and $C(X,d) =1$
then $(X,d)$ is Gromov hyperbolic,
see Proposition~\ref{prop:quasiconstantoneCATzero} and
Question~\ref{ques:ishyperbolic}.

The appearance of a possibly positive $\delta$ in a $(\mu,\delta)$-four-point inequality suggests that $C(X,d)$ can be insensitive to small scales.
Indeed, $C(X,d)$ is a rough isometry invariant of $(X,d)$, Corollary~\ref{cor:roughiso}.
Quasi-isometry is a less stringent condition than rough isometry and 
$C(X,d)$ is {\it not} a quasi-isometry invariant of $(X,d)$.
Examples of this phenomenon are given in \S\ref{sec:basic_examples}.

While the quadrilateral constant,  $C_0(X,d)$, is obviously an isometry invariant it is not a rough isometry invariant; moreover,
the constants $C_0(X,d)$  and $C(X,d)$ need not coincide.
For example,
if $(H^2,d_H)$ is the hyperbolic plane then  $C(H^2,d_H) = 1 < \sqrt{2} = C_0(H^2,d_H)$,
see Example~\ref{ex:hyperbolicspace}.
The intuition supporting this example is that very small quadrilaterals in $H^2$ are approximately Euclidean 
and contribute to  $ C_0(H^2,d_H)$
but not to $ C(H^2,d_H)$.
For spaces $(X,d)$ that are ``four-point scalable in the large'' 
(which we abbreviate as ``scalable'', see Definition~\ref{def:4pointscalable})
we show, Proposition~\ref{prop:4pointscalable}, that $C_0(X,d) = C(X,d)$.
Examples of such spaces include Banach spaces and their metric snowflakes.

A $\CAT(0)$-space is a geodesic metric space whose geodesic triangles are not fatter than corresponding comparison triangles in the Euclidean plane.
Simply connected,
complete Riemannian manifolds of non-positive sectional curvature are familiar examples of  $\CAT(0)$-spaces.
We show, Theorem~\ref{thm:Ptround}, that  the quadrilateral constant  of  a metric space whose distance satisfies Ptolemy's inequality and
the quadrilateral inequality, in particular any $\CAT(0)$-space,  is bounded from above by $\sqrt{2}$. 
The quasi-hyperbolicity constant of any Euclidean space of dimension greater than one is equal to $\sqrt{2}$, Proposition~\ref{prop:euclidean}.

Banach spaces are a particularly important class of metric spaces and
their geometric properties have been extensively studied,
\cite{Johnson-Lindenstrauss}.
For a Banach space $B$ with the metric determined by its norm, we write $C(B)$ for its quadrilateral constant,
which coincides with its quasi-hyperbolicity constant since Banach spaces are scalable.
We observe
that $C(B) \geq J(B)$ where $J(B)$ is the {\it James constant} of $B$, see ~\eqref{eq:James}.
Strong results for the James constant of a Banach space due to Gao and Lau, \cite{Gao-Lau}, and to
Komuro, Saito and Tanaka, \cite{Komuro-Saito}, lead to the following conclusion about $C(B)$.

\begin{thmm}{\rm (Theorem~\ref{thm:banachestimate}) }
If $B$ is a  Banach space with $\dim B >  1$ then $C(B)  \ge\sqrt{2}$.
If $\dim B \geq 3$ and $C(B) = \sqrt{2}$ then $B$ is a Hilbert space.
 \end{thmm}

Enflo \cite{Enflo1969} introduced the notion of the
{\it roundness} of a metric space, Definition~\ref{def:round}, which is a real number greater than or equal to one.
We show:

\begin{thmm}{\rm (Theorem~\ref{thm:round})}
If $B$ is a Banach space with roundness $r(B)$ then $C(B) \leq 2^{1/r(B)}$.
 \end{thmm}
 
 This estimate allows us to calculate the quadrilateral constant of a non-trivial  $L_p$-space.

\begin{corollary}{\rm (Corollary~\ref{cor:lp})}
For a separable measure space $(\Omega, \Sigma, \mu)$ and $1 \leq p \leq \infty$,
let  $L_p(\Omega, \Sigma, \mu)$ be the corresponding $L_p$-space.
If $ \dim L_p(\Omega, \Sigma, \mu) \geq 2$ 
then 
$C(L_p(\Omega, \Sigma, \mu)) =  2^{1/p}$ if $1\le p\le 2$ and $C(L_p(\Omega, \Sigma, \mu)) = 2^{1-1/p}$ if $2\le p\le \infty$. 
\end{corollary}

If $(X,d)$ is any metric space and $0 < \alpha < 1$ then
$(X,d^\alpha)$ is also a metric space, called the {\it $\alpha$-snowflake} 
of $(X,d)$.
We show, Theorem~\ref{thm:snowflake}, that $C_0(X,d^\alpha) \leq 2^\alpha$.  Applying this estimate, we calculate, Proposition \ref {prop:Linfinitysnowflake},
the quadrilateral constant of the $\alpha$-snowflake of $(\IR^n,d_\infty)$, where $d_\infty$ is the $L_\infty$-metric 
(``max metric'') on $\IR^n$:
For $n \geq 2$, $C(\IR^n, d_\infty^\alpha)=2^\alpha$.
Note that the $\alpha$-snowflake  of a Banach space is scalable and so
the quasi-hyperbolicity and quadrilateral constants coincide for such spaces.
The quadrilateral constant of the $\alpha$-snowflake of of the Euclidean line $(\IR^1, d_E)$ can be determined
by solving an associated  optimization problem,
yielding the following calculation.

\begin{thmm}{\rm (Theorem~\ref{thm:lineflake})}
Let $0 < \alpha \leq 1$.
Let $m \geq 1$ be the unique  solution to the equation 
 $(m-1)^\alpha + (m+1)^\alpha =2$.
Then $C_0(\IR^1, d^\alpha_E) = m^{\alpha}$.
\end{thmm}

\section{Quasi-hyperbolicity  and quadrilateral constants}\label{sec:elementary}

We derive basic properties of the quasi-hyperbolicity constant and the quadrilateral constant of a metric space
and examine their general behavior with regard to quasi-isometric embedding and, respectively, bilipschitz embedding.

Recall the following definition from the introduction.

\begin{defi}\label{def:quasi4point}
Let $\mu, \delta \geq 0$.
We say that a metric space $(X,d)$ satisfies the \textit{$(\mu,\delta)$-four-point inequality} if for all $x,y,z,w \in X$,
\[
xy+zw\le \mu\max\{xz+yw, \, xw+yz\}+2\delta.
\]
\end{defi}

We make the following elementary observation concerning this definition.

\begin{prop} \label{prop:basics}
Let $(X,d)$ be a metric space.
\begin{itemize}
\item[(i)]  $(X,d)$ satisfies the $(2, 0)$-four-point inequality,

\item[(ii)] If $(X,d)$ is unbounded and satisfies the $(\mu,\delta)$-four-point inequality then $\mu \geq 1$.

\item[(iii)] If $(X,d)$ is bounded with diameter $D$ then it satisfies the $(0,D)$-four-point \hbox{inequality.}
\end{itemize}
\end{prop}

\begin{proof}
(i).   Let $x, y, z, w \in X$.  Triangle inequality and symmetry of the metric yield:
\[
xy  \le xz+yz, \quad
xy  \le xw+yw, \quad
zw   \le xz+xw, \quad
zw  \le yz+yw.
\]
Adding these four inequalities and dividing by $2$ gives
$xy+zw \le (xz+yw)+(xw+zw)$.
For real numbers $a,b$ we have  $a+b\le 2\max\{a,b\}$ and so
$xy+zw \le 2\max\{xz+yw, \, xw+zw\}$, that is,
the $(2, 0)$-four-point inequality is satisfied. \newline
(ii).  Assume that $X$ is unbounded and satisfies the $(\mu,\delta)$-four-point inequality.
Let $\{x_n\}$ and $\{y_n\}$ be sequences in $X$ such that $x_n y_n \rightarrow \infty$ as $n \rightarrow \infty$.
By the $(\mu,\delta)$-four-point inequality, with $x=x_n$ and $y=z=w=y_n$,
we have
$x_n y_n \leq \mu \, x_n y_n + 2\delta$.
Dividing by $x_n y_n$ and taking the limit as $n \rightarrow \infty$ yields $1 \leq \mu$. \newline
Property (iii) is obvious.
\end{proof}

Given points $x, y, z, w  \in X$, not all identical, define
\begin{equation}\label{eq:Delta}
\Delta(x,y,z,w) =  \frac{xy+zw}{\max\{xz+yw,xw+yz\}}.
\end{equation}
In the introduction, we defined the quadrilateral constant of $(X,d)$ by
\[
C_0(X,d) = \inf\{\mu ~|~  \text{$(X,d)$ satisfies the $(\mu,0)$-four-point inequality}\}.
\]
If $X$ has at least two points then
\begin{equation}\label{eq:Czero}
C_0(X,d) = \sup  \Delta(x,y,z,w) 
\end{equation}
where the supremum is taken over all $x,y,z,w\in X$, not all identical.

We also defined the quasi-hyperbolicity constant of $(X,d)$ by
\begin{align*}\label{eq:qhc}
C(X,d)   = \inf\{\mu ~|~  \mbox{ there exists $\delta \geq 0$}& \mbox{ such that $(X,d)$ satisfies} \\
& \mbox{the $(\mu,\delta)$-four-point inequality}\}. 
\end{align*}


 The quasi-hyperbolicity constant and the quadrilateral constant have the following elementary properties.

\begin{prop}\label{prop:basicstwo} Let $(X,d)$ be a metric space.
\begin{itemize}
\item[(i)]  If  $A \subset X$  and $d_A$ is the subspace metric then $C(A, d_A) \leq C(X,d)$ and  $C_0(A, d_A) \leq C_0(X,d)$.

\item[(ii)]  If $\lambda > 0$ then $C(X, \lambda d) = C(X,d)$ and $C_0(X, \lambda d) = C_0(X,d)$.

\item[(iii)]  $C(X,d) \leq C_0(X,d) \leq 2$.

\item[(iv)] If $(X,d)$ is unbounded then $1 \leq C(X,d)$.

\item[(v)] If $(X,d)$ is bounded then $C(X,d)=0$.

\item[(vi)] If $(X,d)$ has at least two distinct points then $C_0(X,d) \geq 1$.

\item[(vii)] If $\big(X', d' \big)$ is a metric completion of $(X,d)$ then
$C(X,d) = C\big(X', d'  \big)$   and $C_0 (X,d) = C_0\big(X', d'  \big)$.

\end{itemize}
\end{prop}

\begin{proof}
Property (i) and the inequality $C(X,d) \leq C_0(X,d)$ are clear from the definitions of $C(X,d)$ and $C_0(X,d)$. 
Note that for $\lambda >0$, $(X,d)$ satisfies the $(\mu,\delta)$-four-point inequality if and only if $(X,\lambda d)$ satisfies the $(\mu,\lambda\delta)$-four-point inequality.
This implies (ii).
The inequality $C_0(X,d) \leq 2$ in (iii)
is a consequence of Proposition~\ref{prop:basics}(i);
(iv) follows from Proposition~\ref{prop:basics}(ii);
and (v) follows from Proposition~\ref{prop:basics}(iii).
If $x_0, y_0$ are distinct points in $X$
then $\Delta(x_0,y_0,y_0,y_0) = 1$, see~\eqref{eq:Delta},
and so $C_0(X,d) \geq 1$ by ~\eqref{eq:Czero}.
It is straightforward that
a metric space $(X,d)$ satisfies the $(\mu,\delta)$-four-point inequality if and only if a metric completion of $(X,d)$ satisfies the
$(\mu,\delta)$-four-point inequality.
This implies (vii).
\end{proof}

\begin{prop}\label{prop:hyperbolic} Let $(X,d)$ be a metric space.
\begin{itemize}
\item[(i)]  If $(X,d)$ unbounded and Gromov hyperbolic then $C(X,d) =1$.

\item[(ii)]  If $(X,d)$ has at least two points then it is $0$-hyperbolic if and only if $C_0(X,d) =1$. 
\end{itemize}
\end{prop}

\begin{proof}
(i).  By Proposition~\ref{prop:basicstwo}(iv), $C(X,d) \geq 1$.
Since, by definition, a Gromov hyperbolic space satisfies a $(1,\delta)$-four-point inequality for some $\delta \geq 0$ we have  $C(X,d) \leq 1$.
Hence $C(X,d) =1$.\newline
\noindent (ii).  If $(X,d)$ is $0$-hyperbolic then it satisfies the $(1,0)$-four-point inequality and so $C_0(X,d) \leq  1$.
By Proposition~\ref{prop:basicstwo}(vi), $C_0(X,d) \geq 1$. Hence $C_0(X,d) = 1$. 
If $C_0(X,d) =1$  then for every $x, y, z, w \in X$, not  all identical, $\Delta(x,y,z,w) \leq 1$ and so  $(X,d)$
satisfies the $(1,0)$-four-point inequality, that is, $(X,d)$ is $0$-hyperbolic.
\end{proof}

Without additional hypotheses,
the converse of Proposition~\ref{prop:hyperbolic}(i) need not be true,
in \S~\ref{subsec:graph_alpha} we give examples of unbounded metric spaces with $C(X,d)=1$ that are not Gromov hyperbolic
(also see  Question~\ref{ques:ishyperbolic} and Proposition~\ref{prop:quasiconstantoneCATzero}).

\begin{defi}\label{def:4pointscalable}
We say that a metric space $(X,d)$ is {\it four-point scalable in the large}, abbreviated as {\it scalable},
if  for every $x_1, x_2, x_3, x_4 \in X$ and for every $\lambda \geq 0$
there exists $x'_1, x'_2, x'_3, x'_4 \in X$ and $\Lambda \geq \lambda$ such that
$d(x'_i, x'_j) = \Lambda \,d(x_i, x_j)$ for $1 \leq i,j \leq 4$.
\end{defi}

\begin{example}\label{ex:4pointscalable}
Let $V$ be a real vector space with a given norm $\| \cdot \|$.
The norm determines a metric on $V$ given by $d(x,y) = \| x -y \|$.
For any $0 < \alpha \leq 1$ the function $d^\alpha$ is also metric on $V$.
The metric space $(V, d^\alpha)$ is called the {\it $\alpha$-snowflake} of $(V,d)$.
Note that
$d^\alpha (\lambda x, \lambda y) = \lambda^\alpha d^\alpha (x,y)$ for any $\lambda >0$
from which it easily follows that $(V,d^\alpha)$ is scalable. 
Let $S \subset V$ be a nonempty subset such that $\lambda x \in S$ for all $\lambda > 0$ and all $x \in S$.
Then $S$, viewed as a metric subspace of $(V,d^\alpha)$, is also scalable. 
\end{example}

\begin{prop}\label{prop:4pointscalable}
If  $(X,d)$ is scalable 
then $C(X,d)  = C_0(X,d)$.  
\end{prop}

\begin{proof}
It suffices to show that if $(X,d)$ satisfies the $(\mu,\delta)$-four-point inequality for a particular $(\mu,\delta)$  then
it also satisfies the $(\mu,0)$-four-point inequality.
Assume that $(X,d)$ satisfies the $(\mu,\delta)$-four-point inequality for some $\mu\ge 1$ and $\delta\ge 0$. Let $x_1,x_2,x_3,x_4\in X$. 
For each $\lambda\ge 0$,  let $\Lambda\ge\lambda$ and $x'_i\in X$ be such that $x'_ix'_j=\Lambda \, x_ix_j$, $1\le i,j\le 4$.
Note that the $(\mu,\delta)$-four-point inequality for the points $\{x'_i\}$ implies the $(\mu,\delta/\Lambda)$-four-point inequality for $\{x_i\}$.
Since $\Lambda$ can be chosen to be arbitrarily large, it follows that
$\{x_i\}$ satisfies the $(\mu,0)$-four-point inequality.
\end{proof}

\begin{cor} \label{cor:scalable}
Let $V$ be a real vector space with a given norm $\| \cdot \|$ and corresponding
metric, $d(x,y) = \| x -y \|$. 
Let $S \subset V$ be a nonempty subset such that $\lambda x \in S$ for all $\lambda > 0$ and all $x \in S$.
Then for all $0 < \alpha \leq 1$, $C(S, d^\alpha)  = C_0(S, d^\alpha)$.  
\end{cor}

\begin{proof}
From Example~\ref{ex:4pointscalable}, $(S,d^\alpha)$ is scalable 
and so the conclusion follows from Proposition~\ref{prop:4pointscalable}.
\end{proof}

\begin{example}[Hyperbolic space]\label{ex:hyperbolicspace}
Let $n >1$ be an integer and let
$(H^n, d_H)$ denote $n$-dimensional real hyperbolic space.
For this space,
$C(H^n, d_H) = 1  <  \sqrt{2} = C_0(H^n, d_H)$
and so
Proposition~\ref{prop:4pointscalable} implies $(H^n, d_H)$ is {\it not} scalable.  
The space
$(H^n, d_H)$ is Gromov hyperbolic and unbounded, hence
$C(H^n, d_H) =1$ by Proposition~\ref{prop:hyperbolic}(i).
Since $H^n$ has negative sectional curvature as a Riemannian manifold,
$C_0(H^n, d_H) = \sqrt{2}$ by Corollary~\ref{cor:nonnegsectional}.
\end{example}

\begin{defi}\label{def:quasiembedding}
Let $C_1, C_2 >0$ and $L_1,L_2 \geq 0$.
A map $f \colon X \rightarrow Y$ between metric spaces $(X,d_X)$ and $(Y,d_Y)$ is a
{\it $((C_1,L_1), (C_2,L_2))$-quasi-isometric embedding} if for all $u,v \in X$,  
\[
C_1d_X(u,v) - L_1 ~\leq~ d_Y(f(u),f(v)) ~\leq~ C_2d_X(u,v)  + L_2.
\]
The ratio $C_2/C_1$ is called the {\it distortion parameter} and $\max\{L_1/C_1, L_2/C_1\}$ is called the {\it roughness parameter}.

Some useful special cases of this definition include:
\begin{itemize}

\item[(i)]  A $((C_1,0), (C_2,0))$-quasi-isometric embedding $f \colon X \rightarrow Y$ is also known as
               a {\it $(C_1,C_2)$-bilipschitz embedding}.  

\item[(ii)]  A $((1,k),(1,k))$-quasi-isometric embedding $f \colon X \rightarrow Y$ is also known as
               a   {\it $k$-rough isometric embedding}. 
               This condition is equivalent to: for all $u,v\in X$,   $| d_Y(f(u),f(v)) - d_X(u,v) | \leq k$.          
\end{itemize}
\end{defi}

\begin{lemma}\label{lem:quasi_bound}
If $f \colon X \rightarrow Y$ is a $((C_1,L_1), (C_2,L_2))$-quasi-isometric embedding  between metric spaces and
$(Y,d_Y)$ satisfies the $(\mu,\delta)$-four-point inequality for some $(\mu,\delta)$ then
$(X,d_X)$ satisfies the $\left(\tfrac{C_2}{C_1}\mu, \, \tfrac{1}{C_1} (\mu L_2 + L_1 + \delta)\right)$-four-point inequality.
\end{lemma}
 
\begin{proof}
Let $x,y,z,w\in X$ and let $\bar{x},\bar{y},\bar{z},\bar{w}\in Y$ be their respective images under $f\colon X\to Y$. Then 
\begin{align*}
& d_X(x,y) +d_X(z,w) \\
              & \le \tfrac{1}{C_1}\big(d_Y(\bar{x},\bar{y})+d_Y(\bar{z},\bar{w})\big)+\tfrac{2L_1}{C_1}\\
 & \le \tfrac{1}{C_1}\big(\mu\max\left\{d_Y(\bar{x},\bar{z})+d_Y(\bar{y},\bar{w}), \, d_Y(\bar{x},\bar{w})+d_Y(\bar{y},\bar{z})\right\}+2\delta\big)+\tfrac{2L_1}{C1}\\
 & \le \tfrac{1}{C_1}\mu\max\left\{C_2\big(d_X(x,z)+d_X(y,w)\big)+2L_2, \, C_2\big(d_X(x,w)+d_X(y,z)\big)+2L_2\right\}+\tfrac{2\delta}{C_1}+\tfrac{2L_1}{C_1}\\
 & = \tfrac{C_2}{C_1}\mu\max\left\{d_X(x,z)+d_X(y,w),\, d_X(x,w)+d_X(y,z)\right\} + \tfrac{2\mu L_2}{C_1}+\tfrac{2\delta}{C_1}+\tfrac{2L_1}{C_1}
\end{align*}
which shows that $(X,d)$ satisfies the $\left(\tfrac{C_2}{C_1}\mu, \, \tfrac{1}{C_1} (\mu L_2 + L_1+\delta)\right)$-four-point inequality.
\end{proof}

Lemma~\ref{lem:quasi_bound} has the following immediate consequence.

\begin{prop}\label{prop:quasi_bound} 
Let $f \colon X \rightarrow Y$ be a map between metric spaces $(X,d_X)$ and $(Y,d_Y)$. 
\begin{itemize}
\item[(i)]  If $f$ is a $((C_1,L_1), (C_2,L_2))$-quasi-isometric embedding then 
              \[C(X,d_X) \leq (C_2 /C_1) \, C(Y,d_Y). \phantom{\qed}\]

\item[(ii)] If $f$ is a $(C_1,C_2)$-bilipschitz embedding then
              \[C_0(X,d_X) \leq (C_2 /C_1) \, C_0(Y,d_Y). \qed \] 
\end{itemize}
\end{prop}

A map $f \colon X \rightarrow Y$ between metric spaces $(X,d_X)$ and $(Y,d_Y)$ is a 
{\it rough isometry}  if it is a $k$-rough  isometric embedding for some $k \geq 0$ and there exists $R > 0$
such that $f(X)$ is $R$-dense in $Y$, that is, for every $y \in Y$ there exists $x \in X$
such that $d_Y(f(x),y) < R$.
Two metric spaces are {\it  roughly isometric} if there exists a rough isometry between them.
Note that rough isometry is a generally a stronger condition than {\it quasi-isometry}.
Recall that $f$ is a quasi-isometry if it is a
$((C_1,L_1), (C_2,L_2))$-quasi-isometric embedding for some $(C_1, L_1), (C_2, L_2)$
and also $f(X)$ is $R$-dense for some~$R$.

\begin{cor} \label{cor:roughiso}
If $(X,d_X)$ and $(Y,d_Y)$ are roughly isometric then $C(X,d_X) = C(Y,d_Y)$.
\end{cor}

\begin{proof}
It is well-known that if $f \colon X \rightarrow Y$  is a rough isometry then there exists a rough isometry $g \colon Y \rightarrow X$.
Applying Proposition~\ref{prop:quasi_bound}(i) to both $f$ and $g$ yields the conclusion of the Corollary.
For the convenience of the reader, we include a proof of the existence of $g$.
Let $f \colon X \rightarrow Y$ be a $k$-rough isometric embedding such that $f(X)$ is $R$-dense in $Y$.
Define $g \colon Y \rightarrow X$ as follows.
For each $y \in Y$ we can choose $x \in X$ such that $d_Y(f(x), y) < R$ and declare $g(y) = x$.
Observe that for all $y \in Y$, $d_Y(f(g(y)), y) < R$.
For all $u,v \in Y$, $|d_Y(f(g(u)), f(g(v))) - d_X(g(u),g(v))| \leq k$.
Hence, for all $u,v \in Y$, $|d_Y(u, v) - d_X(g(u),g(v))| \leq k + 2R$
and so  $g$ is a $(k + 2R)$-rough embedding.
\end{proof}



\section{Two families of examples}\label{sec:basic_examples}

In \S~\ref{subsec:graph_V}, we exhibit spaces that
are quasi-isometric to the Euclidean line
yet with quasi-hyperbolicity constants that are greater than one
and, consequently, are not Gromov hyperbolic.
In  \S~\ref{subsec:graph_alpha},
we give examples of metric spaces whose quasi-hyperbolicity constants are equal to one, yet are not Gromov hyperbolic.
However, these are examples are not roughly geodesic.
We show,
using Bridson's ``Flat Plane Theorem'', 
that a proper $\CAT(0)$-space whose quasi-hyperbolicity constant is equal to one is necessarily 
Gromov hyperbolic,
see Proposition~\ref{prop:quasiconstantoneCATzero}.


\subsection{The graph of $y=m|x|$ in the Euclidean plane}\label{subsec:graph_V}  
Let $m \geq 0$.  Consider the space $X_m=\{(x,y)\in\IR^2\mid y=m\,|x|\}$ as a subspace of the Euclidean plane.
The metric on $X_m$ is given by
\[
d_E\left((u,m\,|u|),(v,m\,|v|)\right)=\left[(u-v)^2+m^2(|u|-|v|)^2\right]^{1/2}.
\]
Let $(\IR,d_E)$ be the Euclidean line,  $d_E(u,v)=|u-v|$. 
Let  $p\colon X_m\to \IR$ be projection to the first coordinate, that is,  $p(x,y)=x$.
For $u, v \in \IR$,  $|\,|u|-|v|\,| \leq |u-v|$, and so, for $u \neq v$,
\[
\left[(u-v)^2+m^2(|u|-|v|)^2\right]^{1/2}  = |u-v| \left[1+m^2\left(\frac{|u|-|v|}{u-v}\right)^2\right]^{1/2}\le (m^2+1)^{1/2}\,|u-v|,   
\]
and thus for all $u,v$
\[
(m^2+1)^{-1/2} d_E\left((u,m\,|u|),(v,m\,|v|)\right) \le |u-v|\le d_E\left((u,m\,|u|),(v,m\,|v|)\right).   
\]
Hence $p$ is a $((m^2+1)^{-1/2} ,1)$-bilipschitz embedding of $X_m$ into $\IR$.
Since $p$ is surjective, it is also a bilipschitz homeomorphism.
In particular, $(X_m, d_E)$ and $(\IR,d_E)$ are quasi-isometric.

Note that, 
since $(\IR,d_E)$ is $0$-hyperbolic, we have $C(\IR,d_E) =C_0(\IR,d_E)=1$ by Proposition~\ref{prop:hyperbolic}.

For $m >0$,
let $\mu_m = \frac{\sqrt{m^2+1}+1}{\sqrt{m^2+1}-1}$. 
A straightforward calculation yields 
\[
\Delta\left((-\mu_m, \mu_m\, m), \, (1,m), \, (-1,m), \, (\mu_m, \mu_m\,m)\right) = \left(2- (m^2+1)^{-1}\right)^{1/2}   
\]   
and so $C_0(X_m,d_E) \geq \left(2- (m^2+1)^{-1}\right)^{1/2}$.
Note that if $(x,y) \in X_m$ and $\lambda >0$ then  $\lambda(x,y) \in X_m$ and so  Corollary~\ref{cor:scalable} gives $C_0(X_m,d_E)=C(X_m,d_E)$.
Hence $C(X_m,d_E) \geq \left(2- (m^2+1)^{-1}\right)^{1/2} >1$ for $m>0$.
It follows from Proposition~\ref{prop:hyperbolic}(i)  that $(X_m,d_E)$ is not Gromov hyperbolic when $m >0$.
Combining 
Propositions~\ref{prop:quasi_bound} and ~\ref{cor:CATzero} yields the non-sharp upper bound:
\[
C(X_m, d_E) \leq \min\left\{ \sqrt{2},  \, (m^2+1)^{1/2}\right\}.
\]
However, numerical calculations strongly suggest that the configuration
$(-\mu_m, \mu_m\, m)$, $(1,m)$, $(-1,m)$,  $(\mu_m, \mu_m\,m)$
of four points in $X_m$ is optimal, that is,
$C(X_m,d_E) =\left(2- (m^2+1)^{-1}\right)^{1/2}$ for all $m > 0$.


\subsection{The graph of $y=x^\alpha$, where $0 < \alpha  <1$, in the Euclidean plane}\label{subsec:graph_alpha} 

For $0 < \alpha < 1$, let $d_\alpha$ be the metric on the half-line, $[0,\infty)$, given by
\[
d_\alpha(x,y) = \left( (x-y)^2 + (x^\alpha - y^\alpha)^2 \right)^{1/2}.
\]

Let $Y_\alpha=\{(x,y)\in\IR^2 ~\mid~ y=x^\alpha, \, x \geq 0\}$ as a subspace of the Euclidean plane.
Projection to the first coordinate, $(x, y) \mapsto x$, gives an isometry $(Y_\alpha, d_E) \rightarrow  ([0,\infty), d_\alpha)$.
The metric behavior of $([0,\infty), d_\alpha)$ separates into two distinct cases,  namely $0 < \alpha \leq 1/2$ and $1/2 < \alpha < 1$.

\begin{prop}\label{prop:lessthanorequaltoonehalf} 
If $0 < \alpha \leq 1/2$ then for all $x,y \geq 0$, $0 \leq d_\alpha(x,y) - |x-y| \leq 1$.
Consequently, for $0 < \alpha \leq 1/2$,  $([0,\infty), d_\alpha)$ is roughly isometric to the Euclidean half-line and  is thus Gromov hyperbolic.
\end{prop}

\begin{proof}
We first show that if $0 < \alpha \leq 1/2$ then for $u \geq 0$, $\left(u^2 + u^{2\alpha}\right)^{1/2} - u \leq 1$.
Note that for $u\geq 0$ and $0 < \alpha \leq 1/2$ we have $u^{2\alpha} \leq \max\{u^\alpha, u\}$ and so
\[
\left(u^2 + u^{2\alpha}\right)^{1/2} - u  = \frac{u^{2\alpha}}{\left(u^2 + u^{2\alpha}\right)^{1/2} + u} ~\leq~ \frac{\max\{u^\alpha, u\}}{\left(u^2 + u^{2\alpha}\right)^{1/2} + u} ~\leq~ 1.
\]
For $0 < \alpha < 1$ and $x,y \geq 0$,  $| x^\alpha - y^\alpha | \leq |x -y|^\alpha$.
Hence for $0 < \alpha \leq 1/2$ and  $x, y \geq 0$,  and using the inequality
$\left(u^2 + u^{2\alpha}\right)^{1/2} - u \leq 1$ with $u = |x-y|$, we have
\[
0 \leq \left( (x-y)^2 + (x^\alpha - y^\alpha)^2 \right)^{1/2}  - |x-y| \leq \left( |x-y|^2 + |x -y|^{2\alpha} \right)^{1/2} - |x-y| \leq 1,
\]
establishing the conclusion of the Proposition.
\end{proof}

In \cite[1.23 Exercise, p.412]{Bridson-Haefliger} it is asserted that 
$([0,\infty), d_{1/2})$ is not Gromov hyperbolic.  
This is not accurate as demonstrated by Proposition~\ref{prop:lessthanorequaltoonehalf},  however, we show in Proposition~\ref{prop:notgromovhyperbolic}  that
$([0,\infty), d_{\alpha})$ is not Gromov hyperbolic if $1/2 < \alpha <1$.

\begin{lemma}\label{lem:gromov-busting} 
Let $f(\alpha) = 2^{2\alpha-2} + \tfrac{1}{6} (1-2^{2\alpha})^2 - \tfrac{1}{2} (1-2^\alpha)^2 - 2^{4\alpha-3}$. 
If  $0 < \alpha < 1$ then $f(\alpha) > 0$ and
\[
\lim_{t \rightarrow \infty} \left( d_{\alpha}(t,4t) + d_{\alpha}(0,2t) - d_{\alpha}(t,2t) - d_{\alpha}(0,4t) \right)  \, / \,  t^{2\alpha -1}  = f(\alpha).
\]
\end{lemma}

\begin{proof}
Consider the polynomial $g(x) =  \tfrac{1}{24} x^4 -  \tfrac{7}{12} x^2 + x - \tfrac{1}{3} =\tfrac{1}{24} (x -2)^2 (x^2 + 4x -2)$.
Using the factored expression for $g(x)$, we see that $g(x) > 0$ for $1 < x < 2$.
Note that $f(\alpha) = g(2^{\alpha})$.  
Hence $f(\alpha)  > 0$ for $0 < \alpha < 1$.  
For $s \geq 0$, let 
\[
h(s) =  \left( 3^2 + (1-4^\alpha)^2 s \right) ^{1/2} + \left(2^2 + 2^{2\alpha} s\right)^{1/2} - \left( 1 + (1-2^\alpha)^2 s \right)^{1/2} - \left( 4^2 + 4^{2\alpha} s \right)^{1/2}.
\]
A straightforward calculation reveals that, for $t>0$,
\[
\theta(t) ~=~ \left( d_{\alpha}(t,4t)  + d_{\alpha}(0,2t) - d_{\alpha}(t,2t) - d_{\alpha}(0,4t)  \right) \, / \,  t^{2\alpha -1} =  h(t^{2\alpha -2}) \, / \, t^{2\alpha -2}.
\]
Since $2 \alpha - 2 < 0$, $\lim_{t \to \infty} t^{2\alpha -2} =0$ and so
\[
\lim_{t \to \infty} \theta(t) = \lim_{s \to 0 } \frac{h(s)}{s} = h'(0) = f(\alpha)
\]
yielding the conclusion of the Lemma.
\end{proof}

\begin{prop}\label{prop:notgromovhyperbolic} 
If $1/2 < \alpha <1$ then $([0,\infty), d_\alpha)$ is not Gromov hyperbolic.
\end{prop}

\begin{proof}
For $x,y,z,w \in [0,\infty)$ and $0 < \alpha <1$, let
\[
\operatorname{Gr}_\alpha(x,y,z,w) ~=~
d_\alpha(x,y) + d_\alpha(z,w)  - \max\left\{ {d_\alpha(x,z)  + d_\alpha(y,w), d_\alpha(x,w)  + d_\alpha(y,z)} \right\}.
\]
Note that $([0,\infty), d_\alpha)$ is not Gromov hyperbolic if and only if \hbox{$\sup_{x,y,z,w}\operatorname{Gr}_\alpha(x,y,z,w)  = \infty$.}

For $t > 0$, let
\[
h(t) = \frac{d_{\alpha}(t,2t) + d_{\alpha}(0,4t)}{d_{\alpha}(0,t) + d_{\alpha}(2t,4t)}
=\frac{\left(1 + (1 - 2^\alpha)^2 \, t^{2\alpha -2}\right)^{1/2} + \left(4^2+  4^{2\alpha} \, t^{2\alpha -2}\right)^{1/2}}
{\left(1 + t^{2\alpha -2}\right)^{1/2} + \left(2^2 + (2^\alpha - 4^\alpha)^2 \, t^{2\alpha -2}\right)^{1/2}}.
\]
Since $2 \alpha - 2 < 0$, $\lim_{t \rightarrow \infty} t^{2 \alpha - 2}  = 0$ and so the above expression for $h(t)$ yields $\lim_{t \rightarrow \infty} h(t) =  5/3$.
Hence $d_{\alpha}(t,2t) + d_{\alpha}(0,4t) > d_{\alpha}(0,t) + d_{\alpha}(2t,4t)$ for sufficiently large $t$
which implies that 
$\operatorname{Gr}_\alpha(t,4t,0,2t) = d_{\alpha}(t,4t) + d_{\alpha}(0,2t) - d_{\alpha}(t,2t) - d_{\alpha}(0,4t)$
for sufficiently large $t$.
If $1/2 < \alpha <1$ then $2 \alpha - 1 > 0$ and so 
Lemma~\ref{lem:gromov-busting} implies that $\lim_{t \rightarrow \infty}  \operatorname{Gr}_\alpha(t,4t,0,2t) = \infty$. 
\end{proof}

\begin{prop}\label{prop:quasiconstantone} 
If $0 < \alpha <1$ then $C([0,\infty), d_\alpha)=1$.
\end{prop}

\begin{proof}
Let $L >0$.  If $x,y \geq 0$ and $| x -y | \geq L$ then
\[
\frac{| x^\alpha - y^\alpha |}{| x -y|} ~\leq~ \frac{| x - y |^{\alpha}}{| x -y|} ~=~ | x - y |^{\alpha -1} ~\leq~ L^{\alpha -1} ,
\]
and so for $| x - y| \geq L$,
\[
d_\alpha(x,y) \leq \left( |x-y|^2 + \left(L^{\alpha -1}|x-y|\right)^2 \right)^{1/2} \leq \left( 1+ L^{2\alpha -2} \right)^{1/2}|x-y|.
\]
If  $x,y \geq 0$ and $| x -y | \leq L$ then
\[
d_\alpha(x,y) \leq \left( |x-y|^2 + |x -y|^{2\alpha} \right)^{1/2} \leq \left( L^2 + L^{2\alpha} \right)^{1/2} = L \left( 1+ L^{2\alpha -2} \right)^{1/2}.
\]
It follows that  for all $x,y \geq 0$
 \begin{equation}\label{eq:gromovalpha}
|x -y|\leq d_\alpha(x,y) \leq \left( 1+ L^{2\alpha -2} \right)^{1/2}|x-y| +  L \left( 1+ L^{2\alpha -2} \right)^{1/2}.
\end{equation}
Let $d_E(x,y)= |x -y|$, the Euclidean metric on $[0,\infty)$.
By Proposition~\ref{prop:hyperbolic}(i), $C([0,\infty), d_E) = 1$.
Proposition~\ref{prop:quasi_bound}(i) and~\eqref{eq:gromovalpha}
imply that $C([0,\infty), d_\alpha) \leq \left( 1+ L^{2\alpha -2} \right)^{1/2}$.
Since $2\alpha -2 < 0$,  
we have that $\lim_{L \rightarrow \infty} \left( 1+ L^{2\alpha -2} \right)^{1/2} = 1$.
Hence $C([0,\infty), d_\alpha) \leq 1$.
Furthermore,
by Proposition~\ref{prop:basicstwo}(iv), $C([0,\infty), d_\alpha) \geq 1$ and so $C([0,\infty), d_\alpha) =1$.
\end{proof}

\begin{rem} \label{rem:quasialpha}
It follows from
the inequality~\eqref{eq:gromovalpha} that the identity map $([0,\infty), d_E) \rightarrow ([0,\infty), d_\alpha)$ is a quasi-isometry.
In this inequality, there is a trade-off between the
distortion parameter,  $(1+ L^{2\alpha -2})^{1/2}$, and the  roughness parameter,
$L \left( 1+ L^{2\alpha -2} \right)^{1/2}$, that is, an attempt to adjust the number $L$ to make the distortion small (close to $1$) makes
the roughness large and vice versa.
\end{rem}

We showed that for $1/2 < \alpha <1$ the space $([0,\infty), d_\alpha)$ is not Gromov hyperbolic but, nevertheless, $C([0,\infty), d_\alpha)=1$.

\begin{question} \label{ques:ishyperbolic}
Assume that $(X,d)$ is a geodesic metric space or, more generally, roughly geodesic.  Does $C(X,d) =1$ imply that $(X,d)$ is Gromov hyperbolic?
\end{question}

For $1/2 < \alpha <1$, the space $([0,\infty), d_\alpha)$ is not roughly geodesic
and so does not provide a negative answer to this question.
Some evidence in favor of an affirmative answer to Quesition~\ref{ques:ishyperbolic} is given by the following result
(see \S~\ref{sec:catzero} for a discussion of $\CAT(0)$-spaces).

\begin{prop}\label{prop:quasiconstantoneCATzero} 
Let $(X,d)$ be a proper $\CAT(0)$-space.  If  $C(X,d) =1$ then $(X,d)$ is Gromov hyperbolic.
\end{prop}

\begin{proof}
Assume the proper $\CAT(0)$-space $(X,d)$  is not Gromov hyperbolic.
Bridson's {\it Flat Plane Theorem}, \cite[Theorem A]{Bridson1995}, asserts  that there exists an isometric embedding of a Euclidean plane, $(V,d_E)$, into $X$.
Hence $C(V,d_E) \leq C(X,d)$.
By Proposition~\ref{prop:euclidean}, $C(V,d_E)=\sqrt{2}$  and so  $C(X,d) \geq \sqrt{2}$.
In particular, $C(X,d) \neq1$.
\end{proof}


\section{The Ptolemy and quadrilateral inequalities, CAT$(0)$-spaces}\label{sec:catzero}

The notion of a $\CAT(0)$-space generalizes the concept of a simply connected, complete Riemannian manifold of non-positive sectional curvature to  geodesic metric spaces.
We show that the quadrilateral constant of a $\CAT(0)$-space is bounded from above by $\sqrt{2}$.
Indeed, the quadrilateral constant  of any metric space whose distance satisfies Ptolemy's inequality and
the quadrilateral inequality, in particular  any $\CAT(0)$-space, is bounded from above by $\sqrt{2}$,  Theorem~\ref{thm:Ptround}.
The  quadrilateral constant of any Euclidean space of dimension greater than one is equal to $\sqrt{2}$, 
Proposition~\ref{prop:euclidean}.

\begin{defi} \label{prop:basicineq} Let $(X,d)$ be a metric space.
\begin{itemize}\label{ineq:Ptolemy}
\item[(i)] The metric $d$  satisfies {\it Ptolemy's inequality} if for all $x,y,z,w \in X$,
\[
(xy)(zw)\le(xz)(yw)+(xw)(yz).
\]
In this case we say $(X,d)$ is {\it Ptolemaic}.
\item[(ii)]  The metric $d$  satisfies the {\it quadrilateral inequality} if for all $x,y,z,w \in X$,
\[
(xy)^2+(zw)^2\le(xz)^2+(yw)^2+(xw)^2+(yz)^2.
\]
In this case we say $(X,d)$ is {\it $2$-round} (see Definition~\ref{def:round}).
\end{itemize}
\end{defi}

Recall that a {\it Euclidean space} is a real vector space $V$ together with a positive definite inner product, $(u,v) \mapsto \langle u, v \rangle$.
The inner product yields a {\it Euclidean norm}, $\| x \| = \langle x, x \rangle^{1/2}$,  and  a corresponding {\it Euclidean metric}, $d(u,v) = \| x -y\|$.
It is classical mathematics that a Euclidean space with its Euclidean metric is Ptolemaic and $2$-round.

\begin{thm}\label{thm:Ptround}
If the metric space $(X,d)$ is Ptolemaic and $2$-round then 
$C_0(X,d)\le\sqrt{2}$. 
\end{thm}

\begin{proof}
Assume $(X,d)$ is Ptolemaic and $2$-round.   Then for  $x,y,z,w \in X$,
\begin{eqnarray*}
	(xy)(zw) &\le& (xz)(yw)+(xw)(yz) \quad \text{ and } \\
	 (xy)^2+(zw)^2 &\le& (xz)^2+(yw)^2+(xw)^2+(yz)^2.
\end{eqnarray*}
Multiplying the first inequality by $2$ and adding it to the second one yields:
\[
 (xy +zw)^2 \le (xz +yw)^2+ (xw +yz)^2.
\]
For non-negative real numbers $a,b$ we have $\sqrt{a^2 + b^2} \leq \sqrt{2} \, \max\{a,b\}$ and so
the above inequality implies
\[
 xy +zw  \le \sqrt{2} \, \max\{xz +yw, \, xw +yz\}
\]
from which it follows that $C_0(X,d)\le \, \sqrt{2}$.
\end{proof}

Informally, a $\CAT(0)$-space is a geodesic metric space whose geodesic triangles are 
are not fatter than corresponding comparison triangles in the Euclidean plane,
see \cite[II.1.1, page 158]{Bridson-Haefliger} for the precise definition. 
Since any configuration of four points in a  $\CAT(0)$-space has a ``subembedding'' into Euclidean space, \cite[page 164]{Bridson-Haefliger}, 
a $\CAT(0)$-space is Ptolemaic and $2$-round.

\begin{cor}\label{cor:CATzero}
If $(X,d)$ is a subspace of a $\CAT(0)$-space then $C_0(X,d) \leq \sqrt{2}$.
\end{cor}

\begin{proof}
Since a $\CAT(0)$-space is Ptolemaic and $2$-round, so is any subspace.  The conclusion follows from Theorem~\ref{thm:Ptround}.
\end{proof}

\begin{prop}\label{prop:euclidean}
 Let $V$ be a Euclidean space and $d$ its Euclidean metric.  
If $\dim V \geq 2$ then $C(V,d) = C_0(V,d) = \sqrt{2}$.
\end{prop}

\begin{proof}
By Theorem~\ref{thm:Ptround},  $C_0(V,d) \leq \sqrt{2}$.
Since $\dim V \geq 2$,  there are orthogonal unit vectors $u, v \in V$.
A calculation using the inner product of $V$ yields $\Delta(u,v,0,u+v) = \sqrt{2}$ and thus $C_0(V,d) \geq \sqrt{2}$.
Hence  $C_0(V,d) = \sqrt{2}$.
Also, by Corollary~\ref{cor:scalable}, $C(V,d) = C_0(V,d)$.
\end{proof}

Remarkably, a geodesic metric space that is $2$-round is necessarily a $\CAT(0)$-space,
\cite{BergNik, Sato}
and so Corollary~\ref{cor:CATzero} yields the following proposition.

\begin{prop}\label{prop:tworound}
Let $(X,d)$ be a geodesic metric space.
If $(X,d)$ is $2$-round then $C_0(X,d) \leq \sqrt{2}$. \qed
\end{prop}

\begin{rem} \label{rem:snowflake}
Let $(X,d)$ be any metric space.
Blumenthal \cite[Theorem 52.1]{Blumenthal} showed that if 
$0<\alpha\le 1/2$ then the $\alpha$-snowflake $(X,d^\alpha)$   
has the property that any four points in it can be isometrically embedded into Euclidean space. 
Hence, in the case  $0<\alpha\le 1/2$,  $(X,d^\alpha)$ is Ptolemaic and $2$-round and so Theorem~\ref{thm:Ptround} implies that $C_0(X,d^\alpha) \leq \sqrt{2}$. An improvement and extension of this estimate is given by \hbox{Theorem~\ref{thm:snowflake}.}
\end{rem}


\section{Banach spaces}\label{sec:Banach}

In contrast to a $\CAT(0)$-space, whose quadrilateral constant is  bounded from above by $\sqrt{2}$,
the quadrilateral constant of
a Banach space $B$ of dimension greater that one is bounded from {\it below} by $\sqrt{2}$
with equality holding, assuming that the dimension of $B$ is at least three,  only when $B$ is a Hilbert space,
see Theorem~\ref{thm:banachestimate}.
This is a consequence of strong results for the James constant of $B$ due to Gao and Lau, \cite{Gao-Lau}, and to
Komuro, Saito and Tanaka, \cite{Komuro-Saito}.
Enflo \cite{Enflo1969} introduced the notion of the
{\it roundness} of a metric space.
We show, Theorem~\ref{thm:round}, that if
$B$ is a Banach space with roundness $r(B)$ then
its quadrilateral  constant is bounded from above by 
$2^{1/r(B)}$ and use this to show that the 
quadrilateral  constant of a non-trivial $L^p$-space, where $1 \leq p \leq \infty$, is $\max\{2^{1/p},2^{1-1/p}\}$,
see Corollary~\ref{cor:lp}.

Let  $B=(V,\| \cdot \|)$ be a real Banach space.
The norm of $B$,  $\| \cdot \|$, yields a metric $d(u,v) = \| u -v \|$ on the real vector space $V$ and we use notation $C(B)$ for $C_0(V,d)$.
Note that by  Corollary~\ref{cor:scalable} we have $C(V,d) = C_0(V,d) = C(B)$, that is, the quasi-hyperbolicity constant and the quadrilateral constant of $(V,d)$ coincide.

Let $1\le p \leq \infty$. 
Recall the {\it $p$-norm} on $\IR^n$, denoted by $\|x\|_p$ for $x=(x_1,\ldots,x_n)\in \IR^n$, is given by
\begin{eqnarray*}
\|x\|_p = \left\{               
        \begin{array}{cl}
                  \left(|x_1|^p+\cdots+|x_n|^p\right)^{1/p} & \mbox{if } 1 \leq p  < \infty, \\
                  & \\
                  \max\left\{|x_1|,\ldots,|x_n|\right\} & \mbox{if } p= \infty.
        \end{array}\right.
\end{eqnarray*}
We write $\ell_p^n = \left(\IR^n, \|\cdot\|_p\right)$ and $d_p (u,v) =  \|u -v\|_p$.
The $p$-norms on $\IR^n$ are related by the following well-known inequality.
If $1\le p\le q \leq \infty$ then for all $x \in \IR^n$
\begin{equation}\label{eq:wellknown}
\|x\|_q ~\leq~ \|x\|_p ~\leq~ n^{1/p-1/q}\|x\|_q
\end{equation}
where, by convention, $1/\infty = 0$.


Note that $\ell^n_2$ is a Euclidean space and so by Proposition~\ref{prop:euclidean},  $C(\ell^n_2) = \sqrt{2}$  for $n\geq2$.

\begin{prop}\label{prop:ell2p}
$C(\ell^2_p)=2^{1/p}$ if $1\le p\le 2$ and $C(\ell^2_p)=2^{1-1/p}$ if $1<p\le \infty$.  
\end{prop}

\begin{proof}
If $1\le p\le 2$ then by~\eqref{eq:wellknown},
$
\|x\|_2 ~\leq~ \|x\|_p ~\leq~  2^{1/p-1/2}\|x\|_2
$.
By Proposition~\ref{prop:quasi_bound}, 
\[
C(\ell^2_p) ~\leq~ 2^{1/p-1/2}\, C(\ell_2^2) ~=~ 2^{1/p}.
\]
Observe
$\Delta((-1,1), \, (1, -1),  \,  (-1,-1),  \, (1,1))= 2^{1/p}$ and so
$C(\ell^2_p) \geq 2^{1/p}$. Thus 
$C(\ell^2_p) = 2^{1/p}$.

If $2\le p \leq \infty$ then by~\eqref{eq:wellknown},
$
2^{1/p-1/2}\|x\|_2 ~\leq~  \|x\|_p ~\leq~  \|x\|_2
$.
By Proposition~\ref{prop:quasi_bound}, 
\[
C(\ell^2_p) ~\leq~ 2^{1/2-1/p}\, C(\ell_2^2) ~=~ 2^{1-1/p}.
\]
Observe
$\Delta((0,1), \, (0, -1),  \,  (-1,0),  \, (1,0))= 2^{1-1/p}$ and so
$C(\ell^2_p) \geq 2^{1-1/p}$. Thus 
 $C(\ell^2_p) = 2^{1-1/p}$.
\end{proof}

Proposition~\ref{prop:ell2p} generalizes to non-trivial $L_p$-spaces, see Corollary~\ref{cor:lp}.

The {\it Banach-Mazur distance} between two isomorphic Banach spaces $E$ and $F$ is defined by
\[
d_{\rm BM}(E,F) = \inf\{ \| T \| \, \|T^{-1} \|  ~|~  \text{$T\colon E \rightarrow F$ is an isomorphism}\}.
\]
For example, if $1\leq p \leq q \leq 2$ or $2 \leq p \leq q \leq \infty$ then
$d_{\rm BM}(\ell^n_p,\ell^n_q) = n^{1/p-1/q}$, \cite[Proposition 37.6]{TJ}.
Proposition~\ref{prop:quasi_bound} yields the following comparison.

\begin{prop}\label{prop:bmdistance}
If $E$ and $F$ are isomorphic Banach spaces
then 
$\,C(E) \, \leq \, d_{\rm BM}(E,F) \,C(F)$. \qed
\end{prop}
Because of Theorem~\ref{thm:banachestimate} below,
the inequality of Proposition~\ref{prop:bmdistance} can only give useful information when
$d_{\rm BM}(E,F) < \sqrt{2}$.

Since, up to a translation, any four points of a Banach space lie in some subspace of dimension at most three,
\begin{equation} \label{eq:subspace}
C(B) = \sup\{ C(V) ~|~ \text{$V$ is a subspace of $B$ with $\dim V \leq 3$}\}.
\end{equation}
A Banach space $B$ is {\it finitely representable} in another Banach space $B'$ if 
for every finite dimensional subspace $F$ of $B$ and every $\varepsilon >0$
there is a subspace $F'$ of $B'$ and an isomorphism $T \colon F \rightarrow F'$
such that $\| T \| \, \|T^{-1} \| \leq 1 + \varepsilon$.
\begin{prop}\label{prop:finiterep}
If $B$ is finitely representable in $B'$ then $C(B) \leq C(B')$.
\end{prop}
\begin{proof}
Let $\varepsilon >0$.
Let $V$ be a subspace of $B$ with $\dim V \leq 3$.
Since $B$ is finitely representable in $B'$,
there exists a subspace $V'$ of $B'$ and an isomorphism
$T \colon V \rightarrow V'$
such that $\| T \| \, \|T^{-1} \| \leq 1 + \varepsilon$.
By Proposition~\ref{prop:quasi_bound},  $C(V) \leq (1 + \varepsilon) \, C(V')$
and so $C(V) \leq (1 + \varepsilon) \, C(B')$ because $C(V') \leq C(B')$.
It follows from~\eqref{eq:subspace} that $C(B) \leq (1 + \varepsilon) \, C(B')$.
Since $\varepsilon$ is arbitrary, we conclude $C(B) \leq C(B')$.
\end{proof}

\begin{cor} \label{cor:seconddual}
Let $B$ be a Banach space and $B^{**}$ its second dual. 
Then 
$C(B) = C(B^{**})$.
\end{cor}
\begin{proof}
The canonical map $B \rightarrow B^{**}$ is an isometric embedding and hence $C(B) \leq C(B^{**})$.
In any Banach space $B$, the second dual $B^{**}$ is finitely representable in in $B$, \cite[\S 9]{Johnson-Lindenstrauss},
and so by Proposition~\ref{prop:finiterep},  $C(B^{**}) \leq C(B)$.
It follows that $C(B) = C(B^{**})$.
\end{proof}

The {\it James constant} of a Banach space $B$ is defined by:
\[
J(B) =  \sup\{ \min( \| x -y \|, \| x + y \| )   ~|~ \,  \| x \| =  \| y \|  = 1\}
\]
If $\|x\| = \|y\| = 1$ then $\Delta(x,y,0,x+y) = \tfrac{1}{2} ( \| x -y \|  + \| x + y \| ) $ and thus
\begin{equation} \label{eq:James}
C(B) \geq  \sup \{ \tfrac{1}{2} ( \| x -y \|  + \| x + y \| )   ~|~ \,  \| x \| =  \| y \|  = 1\}   \geq  J(B)
\end{equation}
A Banach space $B$ is said to be {\it non-trivial} if $\dim(B) \geq 2$.

\begin{thm}\label{thm:banachestimate}
If $B$ is any non-trivial  Banach space then $C(B)  \ge\sqrt{2}$.
If $\dim B \geq 3$ and $C(B) = \sqrt{2}$ then $B$ is a Hilbert space.
\end{thm}
\begin{proof}
Gao and Lau, \cite[Theorem 2.5]{Gao-Lau}, show  $J(B) \geq \sqrt{2}$ for any non-trivial Banach space $B$.
Furthermore, Komuro, Saito and Tanaka, \cite{Komuro-Saito}, show that $\dim B \geq 3$ and $J(B) = \sqrt{2}$ implies $B$ is a Hilbert space.
The conclusion of the theorem follows from~\eqref{eq:James}.
\end{proof}

\begin{defi}[\cite{Enflo1969}]\label{def:round}
Let $(X,d)$ be a metric space and $p\geq1$.
The space $(X,d)$ is said to be {\it $p$-round}
if for all $x,y,z,w \in X$,
$
(xy)^p+(zw)^p \leq (xz)^p+(yw)^p+(xw)^p+(yz)^p.
$
The {\it roundness} of $(X,d)$ is
$r(X,d) = \sup\{ p ~|~  \text{$(X,d)$ is $p$-round}\, \}$.
 \end{defi}
Note that if $r(X,d)<\infty$ then the supremum is attained.  
Enflo, \cite{Enflo1969}, observed that $r(X,d) \geq 1$ and that if  $(X,d)$ has the {\it midpoint property}\footnote{A metric space $(X,d)$ has the {\it midpoint property} 
if for every $x,y \in X$ there exists $z \in X$ such that $d(x,z)=d(z,y)= \tfrac{1}{2}d(x,y)$.}
then $r(X,d) \leq 2$. In particular, if $B$ is a Banach space then $1 \leq r(B) \leq 2$, where $r(B)$ is the roundness of $B$ as a metric space.

\begin{lemma}\label{lem:round}
 Let $B$ be a Banach space that is $p$-round.  Then for any vectors  $e, f \in B$
\[
 \left(\| e \|  + \| f  \|\right)^p \leq  \| e -f \|^p  + \| e+ f \|^p.
 \]
 \end{lemma}
 \begin{proof}
In the ``$p$-round inequality'' of Definition~\ref{def:round},  letting $x = e + f$, $y= e -f$, $w = 2e$, and $z=0$ gives
\[
\| 2e \|^p  + \|2 f  \|^p \leq  2\| e -f \|^p  +2\| e+ f \|^p
\]
and so
\[
2^{p-1} \left(\| e \|^p  + \|f  \|^p\right) \leq  \| e -f \|^p  + \| e+ f \|^p.
\]
By~\eqref{eq:wellknown}, with $n=2$, $ \left(\| e \|  + \| f  \|\right)^p  \leq 2^{p-1} \left(\| e \|^p  + \|f  \|^p\right)$ from which the conclusion follows.
 \end{proof}

 \begin{thm} \label{thm:round}
If $B$ is a Banach space then $C(B) \leq 2^{1/r(B)}$.
\end{thm}

\begin{proof}
Let $p = r(B)$. Then $B$ is $p$-round.
Let  $x,y,z,w \in B$.
Let $a = x -z$, $b=w-y$, $c=w-x$, $d= y-z$, $e = y - x$, and $f= w -z$.
Note that $f = a + c = b+d$ and $e = d -a = c -b$.
Hence $e + f =c + d $ and $f - e = a  +b$.
By Lemma~\ref{lem:round},
\begin{eqnarray*}
  \left(\| e \|  + \| f  \|\right)^p ~\leq   &  \| e -f \|^p  + \| e+ f \|^p  & \null \\
        ~ =  &   \| a + b \|^p  + \| c+ d \|^p & \null  \\
       ~ \leq   & \left( \| a\|  + \|b \|\right)^p   +\left( \| c\|  + \|d \|\right)^p  &\text{  (by triangle inequality).}
\end{eqnarray*}
It follows that
\begin{eqnarray*}
     \| e \|  + \| f  \| ~\leq   &   \left( \left( \| a\|  + \|b \|\right)^p   +\left( \| c\|  + \|d \|\right)^p \right)^{1/p}   & \null  \\
                            ~\leq   &  2^{1/p} \max\left( \| a\|  + \|b \|, \, \| c\|  + \|d \| \right)  & \text{ (by~\eqref{eq:wellknown}).}
\end{eqnarray*}
Thus the 
$(2^{1/p} , 0)$-four-point inequality holds and so 
$C(B) \leq 2^{1/p} $.
\end{proof}

\begin{cor} \label{cor:lp}
Let $(\Omega, \Sigma, \mu)$ be a separable measure space, 
that is, the $\sigma$-algebra $\Sigma$ is generated by a countable collection of subsets of $\Omega$.
Let $1 \leq p \leq \infty$ and
let  $L_p(\Omega, \Sigma, \mu)$ be the corresponding $L_p$-space.
If $ \dim L_p(\Omega, \Sigma, \mu) \geq 2$ 
then $C(L_p(\Omega, \Sigma, \mu)) =  2^{1/p}$ if $1\le p\le 2$ and $C(L_p(\Omega, \Sigma, \mu)) = 2^{1-1/p}$ if $2\le p\le \infty$.  
\end{cor}

\begin{proof}
Denote $B =  L_p(\Omega, \Sigma, \mu)$.
Assume $\dim B \geq 2$.
In the case $1 \leq p \leq 2$, Enflo, \cite{Enflo1969}, showed that $r(B) = p$
and so $C(B) \leq 2^{1/p}$ by Theorem~\ref{thm:round}.
In the case $2 \leq p \leq \infty$, by \cite[Proposition1.4 and Remark 1.5]{LTW},
$r(B) = 1/(1 - 1/p)$
and so $C(B) \leq 2^{1- 1/p}$ by Theorem~\ref{thm:round}.
Hence for $1 \leq p \leq \infty$, $C(B) \leq  \max\{2^{1/p},2^{1-1/p}\}$.

The classification theory of $L_p$ spaces (see \cite[\S4]{Johnson-Lindenstrauss})
gives that, for $1\leq p < \infty$, the space $B =  L_p(\Omega, \Sigma, \mu)$ is isometric to one
of  the Banach spaces in the list
\begin{equation}\label{eq:banachlist}
\ell^n_p, ~~\ell_p, ~~ L_p(0, 1), ~~\ell_p \oplus_p L_p(0, 1), ~~\ell^n_p \oplus_p L_p(0, 1) \qquad n=1,2, \ldots
\end{equation}
Here, $\ell_p$ denotes the space of sequences $(x_n)^\infty_{n=1}$ with $\sum^\infty_{n=1} |x_n|^p < \infty$ and
$L_p(0, 1)$ denotes the space of measurable functions (modulo null sets) on the unit interval such that
$\int^1_0 |f(x)|^p dx < \infty$, and
$\oplus_p$ denotes the $\ell_p$ direct sum,  that is, $\| a \oplus b \| = ( \| a\|^p   + \| b\|^p)^{1/p}$. 
Each of the spaces in the list~\eqref{eq:banachlist} (in the case of  $\ell^n_p$, assume $n \geq 2$)  contains a subspace isometric to $\ell^2_p$ and
so  $C(B) \geq C(\ell^2_p) =  \max\{2^{1/p},2^{1-1/p}\}$ by Proposition~\ref{prop:ell2p}.
Hence $C(B) =  \max\{2^{1/p},2^{1-1/p}\}$.
In the case $p = \infty$ note that $B$ contains a subspace isometric to $\ell^2_\infty$ which implies that $C(B) = 2$. 
\end{proof}

\begin{question} 
Let $(X,d)$ be a geodesic metric space.
Is $C_0(X,d) \leq 2^{1  / r(X,d)}$?
\end{question}
By Proposition~\ref{prop:tworound}, this is true in the case $r(X,d) = 2$.


\section{Snowflaked metric spaces} 

Recall that if $0 < \alpha \leq 1$ and $(X,d)$ is any metric space then
$(X,d^\alpha)$ is also a metric space, called the {\it $\alpha$-snowflake} of $(X,d)$.
We show that $C_0(X,d^\alpha) \leq 2^\alpha$, Theorem~\ref{thm:snowflake},
and give some applications of this estimate.
We determine the quadrilateral constant of the $\alpha$-snowflake of the Euclidean real line, Theorem~\ref{thm:lineflake}.
Recall that the $\alpha$-snowflake  of a Euclidean space is scalable and so
the quasi-hyperbolicity and quadrilateral constants coincide for such spaces.

\begin{lemma}\label{lem:ultra}
Let $a_{ij}\in \IR$, $i,j\in\{1,2,3,4\}$, be such that $a_{ij}=a_{ji}$. Let  $\lambda\ge 1$.
If $a_{ij}\le \lambda \max\{a_{ik},a_{kj}\}$ for all $i,j,k$, then $a_{ij}+a_{k\ell}\le \lambda \max\{ a_{ik}+a_{j\ell}, \,  a_{i\ell}+a_{jk}\}$ for all $i,j,k,\ell$.
\end{lemma}
Note that if $L$, $M$ and $S$ denote the largest, medium and smallest of the three sums $a_{ij}+a_{k\ell}$, $a_{ik}+a_{j\ell}$ and $a_{i\ell}+a_{jk}$ for some choice of $i,j,k,\ell\in\{1,2,3,4\}$,  then the conclusion of the lemma is equivalent to  $L\le \lambda M$.

\begin{proof}
Fix $i,j,k,\ell\in\{1,2,3,4\}$. Without loss of generality, assume that $L=a_{ij}+a_{k\ell}$ is the largest sum and assume that $ a_{k\ell}\le a_{ij}$. Since  
$a_{ij}\le \lambda \max\{a_{ik},a_{kj}\}$ 
 and $a_{ij}\le \lambda \max\{a_{i\ell}, \, a_{\ell j}\}$, we have 
\begin{equation*}a_{ij}+a_{k\ell}\le a_{ij} + a_{ij}\le \lambda \max\{a_{ik}+a_{i\ell},\, a_{ik}+a_{\ell j}, \, a_{kj}+a_{i\ell}, \, a_{kj}+a_{\ell j}\}.\end{equation*}

\noindent If $a_{ik}\ge a_{kj}$ and $a_{\ell j}\ge a_{i\ell}$ then 
$$M=a_{ik}+a_{\ell j} = \max\{a_{ik}+a_{i\ell}, \, a_{ik}+a_{\ell j}, \, a_{kj}+a_{i\ell}, \, a_{kj}+a_{\ell j}\}$$
and if $a_{ik}\le a_{kj}$ and $a_{\ell j}\le a_{i\ell}$ then  $$M=a_{kj}+a_{i\ell} = \max\{a_{ik}+a_{i\ell}, \, a_{ik}+a_{\ell j}, \, a_{kj}+a_{i\ell}, \, a_{kj}+a_{\ell j}\}.$$
In both cases, $L\le \lambda M$. Furthermore, if $a_{ik}\ge a_{kj}$ and $a_{\ell j}\le a_{i\ell}$\, then $a_{ij}\le \lambda \max\{a_{ik}, \, a_{kj}\} = \lambda a_{ik}$ and $a_{ij}\le \lambda \max\{a_{i\ell}, \, a_{\ell j}\} = \lambda a_{i\ell}$, and since $a_{k\ell}\le \lambda \max\{a_{kj}, \, a_{\ell j}\}$,
 \begin{align*}
 a_{ij}+a_{k\ell}& \le a_{ij}+\lambda \max\{a_{kj}, \, a_{\ell j}\} = \max\{a_{ij}+\lambda a_{kj}, \, a_{ij}+\lambda a_{\ell j}\}\\
 &\le \max\{\lambda a_{i\ell}+\lambda a_{kj}, \, \lambda a_{ik}+\lambda a_{\ell j}\} = \lambda \max\{a_{i\ell}+a_{kj}, \, a_{ik}+a_{\ell j}\}.
 \end{align*}
\noindent Finally, if $a_{ik}\le a_{kj}$ and $a_{\ell j}\ge a_{i\ell}$\, then 
\[
a_{ij}\le \lambda \max\{a_{ik}, \, a_{kj}\} = \lambda a_{kj} \,\, \text{   and   } \,\, a_{ij}\le \lambda \max\{a_{i\ell}, \, a_{\ell j}\} = \lambda a_{\ell j},
\]
and since $a_{k\ell}\le \lambda \max\{a_{ki}, \, a_{i\ell}\}$, we have
\[
a_{ij}+a_{k\ell} \le a_{ij}+\lambda \max\{a_{ki}, \, a_{i\ell}\}\le \lambda \max\{a_{\ell j}+a_{ki}, \, a_{kj}+a_{i\ell}\},
\]
that is, $L\le \lambda M$.
\end{proof}

\begin{thm}\label{thm:snowflake}
Let $0<\alpha\le 1$.
For any metric space $(X,d)$,  $C_0(X,d^\alpha)\le 2^\alpha$.
\end{thm}

\begin{proof}
Let $x_i\in X$, $i=1,2,3,4$. It suffices to show that if $i,j,k,l\in\{1,2,3,4\}$ then
\[
(x_ix_j)^\alpha + (x_kx_l)^\alpha \le 2^\alpha \max\{(x_ix_k)^\alpha +(x_jx_l)^\alpha, \, (x_ix_l)^\alpha +(x_jx_k)^\alpha\}.
\]
Observe that for all $i,j,k$ triangle inequality implies $x_ix_j\le x_ix_k + x_jx_k \le 2\max\{x_ix_k,x_jx_k\}$.
Hence 
\[
(x_ix_j)^\alpha \le 2^\alpha\max\{(x_ix_k)^\alpha,(x_jx_k)^\alpha\}.
\] The conclusion follows from Lemma~\ref{lem:ultra} with $a_{ij}=(x_i x_j)^\alpha$ and $\lambda=2^\alpha$.
\end{proof}

As in \S\ref{sec:Banach}, $d_p$, where $1 \leq p \leq \infty$, denotes the metric on $\IR^n$ determined by the standard $p$-norm.

\begin{prop}\label{prop:Linfinitysnowflake}
If $0<\alpha \leq 1$ and $n \geq 2$ then $C_0(\IR^n, d_\infty^\alpha)=2^\alpha$
\end{prop}

\begin{proof}
By Theorem~\ref{thm:snowflake},  $C_0(\IR^n, d_\infty^\alpha) \leq 2^\alpha$. 
Consider the following four points in  $\IR^n$:
\[
x=(0,1, 0, \ldots,0),  ~  y=(0,-1, 0, \ldots,0), ~  z=(-1, 0, \ldots,0),  ~w=(1, 0, \ldots,0).
\]
A calculation using the metric $d_\infty^\alpha$  yields $\Delta(x,y,z,w) = 2^\alpha$ and thus $C_0(\IR^n, d_\infty^\alpha) \geq 2^\alpha$. Hence 
$C_0(\IR^n, d_\infty^\alpha) =   2^\alpha$.
\end{proof}

The same technique gives a non-sharp estimate for $C_0(\IR^n, d_2^\alpha)$, where $n \geq 2$, as follows.

\begin{prop}
If $0<\alpha \leq 1$ and $n \geq 2$ then $2^{\alpha/2} \leq C_0(\IR^n, d^\alpha_2) \leq 2^{\min\{\alpha, \, 1/2\}}$.
\end{prop}

\begin{proof}
By Theorem~\ref{thm:snowflake},  $C_0(\IR^n, d^\alpha_2) \leq 2^{\alpha}$.
Schoenberg showed, \cite[Theorem~1]{Sch:37On-:aa}, that $(\IR^n, d^\alpha_2)$ isometrically embeds into (infinite dimensional) Hilbert space
and hence $C_0(\IR^n, d^\alpha_2) \leq 2^{1/2}$. 
Consequently, $C_0(\IR^n, d^\alpha_2) \leq 2^{\min\{\alpha, \, 1/2\}}$.
For the four points 
$x,y,z,w \in \IR^n$ specified in the proof of  Proposition~\ref{prop:Linfinitysnowflake},
we have
$\Delta(x,y,z,w) = 2^{\alpha/2}$, yielding the lower bound for $C_0(\IR^n, d^\alpha_2)$.
\end{proof}

Numerical calculations suggest the following exact value for $C_0(\IR^n, d^\alpha_2)$ when $n \geq 2$.

\begin{conj}\label{conj:euclideansnowflake}
Let $0 < \alpha < 1$.
If $n \geq 2$ then  $C_0(\IR^n, d^\alpha_2)=2^{\alpha/2}$.
\end{conj}
The $\alpha$-snowflakes of the Euclidean line turns out to be of a different nature than the spaces $(\IR^n, d^\alpha_2)$ with $n \geq 2$,
as revealed in the following theorem.

\begin{thm} \label{thm:lineflake}
Let $0 < \alpha \leq 1$  and $d_E^\alpha(x,y) = |x-y|^\alpha, x,y\in \R$.
Let $m \geq 1$ be the unique  solution to the equation  $(m-1)^\alpha + (m+1)^\alpha =2$.
Then $C_0(\IR^1, d^\alpha_{E}) = m^{\alpha}$.
\end{thm}
We have that
\begin{equation}\label{eq:line_snowflaked}
C_0(\R^1,d_E^\alpha) = \sup\Delta(x,y,z,w), 
\end{equation} where 
\[\Delta(x,y,z,w) = \frac{|x-y|^\alpha+|z-w|^\alpha}{\max\{ |x-z|^\alpha + |y-w|^\alpha, \ |x-w|^\alpha + |y-z|^\alpha\}}\]
and the supremum in (\ref{eq:line_snowflaked}) is taken over all $x,y,z,w\in\R$, not all identical.
Since the map $(x,y,z,w)\mapsto \Delta(x,y,z,w)$ is translation and scale invariant,
we may assume that $x=0$, $y=1+s$, $z=1-t$, and $w=2$, with $(t,s)\in D=\{(t,s) \in [-1,1] \times [-1,1]  ~\mid~ t+s\ge 0\}.$
Then 
\[ 
(t,s) \mapsto \Delta(0, 1+s, 1-t, 2)  = \frac{(1+s)^\alpha+(1+t)^\alpha}{\max\limits_{(t,s)\in D}\{ (1-t)^\alpha+(1-s)^\alpha,(t+s)^\alpha + 2^\alpha\}}
\] 
is continuous on the compact set $D$ and 
\[
C_0(\R^1,d_E^\alpha) = \max_{(t,s)\in D }\Delta(0, 1+s, 1-t, 2).
\]
Furthermore, if $F, G \colon D\to \R$ are given by
\begin{equation}\label{eq:F_and_G}
F(t,s)=\frac{(1+t)^\alpha+(1+s)^\alpha}{(1-t)^\alpha+(1-s)^\alpha} \mbox{ and } G(t,s) = \frac{(1+t)^\alpha+(1+s)^\alpha}{(t+s)^\alpha + 2^\alpha},
\end{equation} 
and $D_1 = \{(t,s)\in D \mid F(t,s)\le G(t,s)\}$ and $D_2 = \{(t,s)\in D \mid F(t,s)\ge G(t,s)\}$,  then 
\[
\Delta(0,1-t,1+s,2) = \min_{(t,s)\in D}\{F(t,s),G(t,s)\} = 
\begin{cases} 
F(t,s), & (t,s) \in D_1\\
G(t,s), & (t,s) \in D_2,
\end{cases}\]
and 
\begin{equation}\label{eq:max_subsets}
C_0(\R^1,d_E^\alpha) = \max\left\{\max_{(t,s)\in D_1}F(t,s), \max_{(t,s)\in D_2}G(t,s)\right\}.
\end{equation}
The following lemma shows that the maximum in (\ref{eq:max_subsets}) is attained on $D_0 = D_1\cap D_2$.

\begin{lemma}\label{lem:max_on_intersection}
Let $0 < \alpha < 1$. Let $F,G\colon D \to \R$ be given by (\ref{eq:F_and_G}) and let $D_0 = \{(t,s)\in D \mid F(t,s) = G(t,s) \}$. Then 
\[
C_0(\R^1,d_E^\alpha) = \max_{(t,s)\in D_0}F(t,s).
\]
\end{lemma}

\begin{proof}
We show that $F$ and $G$ attain their maximum on the boundary of $D_1$ and $D_2$, respectively.  Indeed, the partial derivatives of $F$,
\[ F_t (t,s) = \frac{\alpha(1+t)^{\alpha-1}}{(1-t)^\alpha+(1-s)^\alpha} +\frac{\alpha ((1+t)^\alpha+(1+s)^\alpha)(1-t)^{\alpha-1}}{((1-t)^\alpha+(1-s)^\alpha)^2} \]
\[ F_s (t,s) = \frac{\alpha(1+s)^{\alpha-1}}{(1-t)^\alpha+(1-s)^\alpha} +\frac{\alpha ((1+t)^\alpha+(1+s)^\alpha)(1-s)^{\alpha-1}}{((1-t)^\alpha+(1-s)^\alpha)^2}\]  are defined for all $(t,s)\in (-1,1)^2, t+s>0$ and $F_t >0$ and $F_s >0$.  Thus $\max_{(t,s)\in D_1}F(t,s)$ is attained on the boundary  $\partial D_1 = D_0 \cup \{(t,s)\in D \mid t+s=1\}$. Note that $F(t,s)\ge 1$ for $(t,s)\in D$ and $F(t,s)=1$ if and only if $t+s=1$. Hence 
\begin{equation}\label{eq:F_max}
\max_{(t,s)\in D_1}F(t,s) = \max_{(t,s)\in D_0}F(t,s).
\end{equation}
The partial derivatives of $G$
\[ G_t (t,s) = \frac{\alpha(1+t)^{\alpha-1}}{(t+s)^\alpha +2^\alpha} -\frac{\alpha ((1+t)^\alpha+(1+s)^\alpha)(t+s)^{\alpha-1}}{((t+s)^\alpha+2^\alpha)^2}\]
\[ G_s (t,s)= \frac{\alpha(1+s)^{\alpha-1}}{(t+s)^\alpha +2^\alpha} -\frac{\alpha ((1+t)^\alpha+(1+s)^\alpha)(t+s)^{\alpha-1}}{((t+s)^\alpha+2^\alpha)^2}\]
are defined for all $(t,s)\in (-1,1)^2, t+s > 0$ and $G_t = G_s = 0$ if and only if $t=s=1$. Thus $\max_{(t,s)\in D_2}G(t,s)$ is attained on the boundary $\partial D_2 = D_0 \cup \{(t,s)\in D \mid t=1\} \cup \{(t,s)\in D \mid s=1\} $. 
Note also that $G(t,s)\ge1$ for $(t,s)\in D$ and $G(t,s)=1$ if and only if $t=1$ or $s=1$. Hence 
\begin{equation}\label{eq:G_max}
\max_{(t,s)\in D_2}G(t,s) = \max_{(t,s)\in D_0}G(t,s).
\end{equation}
The conclusion follows from (\ref{eq:max_subsets}) together with (\ref{eq:F_max}) and (\ref{eq:G_max}). 
\end{proof}

The following result shows that $\max_{(t,s)\in D_0}F(t,s)$ is attained when $t=s$.

\begin{lemma}\label{lem:max_on_diagonal}
Let $0<\alpha<1$. Let $F,G\colon D \to \R$ be given by (\ref{eq:F_and_G}) and let $D_0 = \{(t,s)\in D \mid F(t,s) = G(t,s) \}$. Then 
\[ 
\max_{(t,s)\in D_0}F(t,s) = \left(\frac{1+a}{1-a}\right)^\alpha,
\]
where $0< a <1$ is the unique solution of $F(a,a) = G(a,a)$. 
\end{lemma}

\begin{proof}
Notice that if $(t,s)\in D_0$ then $t=-1$ if and only if $s=1$ and $F(-1,1)=1$. By symmetry, $F(1,-1) =1$. Since $F(t,s)\ge 1$ on $D_0$, the maximum of $F|_{D_0}$, the restriction of $F$ to $D_0$, is not attained at $(-1,1)$ or $(1,-1)$. Let $(a,b)\in D_0$, with $a\neq \pm 1$. If $F$ attains a local extremum at $(a,b)$ subject to the constrain $F(t,s)=G(t,s)$, then the level curves $\{(t,s)\in D\mid F(t,s)=F(a,b)\}$ and $\{(t,s)\in D\mid F(t,s) - G(t,s) = 0\}$ are both tangent at $(a,b)$. Since $F_s(a,b)-G_s(a,b)\neq 0$, by the Implicit Function Theorem, there exists an open neighbourhood $U\subseteq (-1,1)$ of $a$ and a function $\omega = \omega(t)$ such that $F(t,\omega(t)) - G(t,\omega(t))=0$ for $t\in U$. Furthermore,
\[\omega'(t) = -\frac{(t+\omega)^{\alpha-1} + (1-t)^{\alpha-1}}{(t+\omega)^{\alpha-1}+(1-\omega)^{\alpha-1}}\] 
for all $t\in U$. 
Similarly, since $F_s(a,b)\neq 0$, there exists an open neighbourhood $V\subseteq (-1,1)$ of $a$ and a function $\nu = \nu(t)$ on $V$ such that $F(t,\nu(t)) = F(a,b)$ on $V$. Also, for all $t\in V$,
\[\nu'(t) = -\frac{(1+t)^{\alpha-1} + F(a,b) (1-t)^{\alpha-1}}{(1+\nu)^{\alpha-1} + F(a,b) (1-\nu)^{\alpha-1}}.\]
Hence, a necessary condition for $(a,b)$ to be a point of local extremum for $F|_{D_0}$ is that $\omega'(a) = \nu'(a)$. Using that $\omega(a) = \nu(a) = b$, that is,
\[
\frac{(a+b)^{\alpha-1} + (1-a)^{\alpha-1}}{(a+b)^{\alpha-1}+(1-b)^{\alpha-1}}=\frac{(1+a)^{\alpha-1} + F(a,b) (1-a)^{\alpha-1}}{(1+b)^{\alpha-1} + F(a,b) (1-b)^{\alpha-1}},
\]
equivalently,
\begin{align*}
(a+b)^{\alpha-1}\left[ (1+b)^{\alpha-1} \right.& \left.+F(a,b)(1-b)^{\alpha-1} - (1+a)^{\alpha-1} - F(a,b)(1-a)^{\alpha-1} \right] \\
& \qquad \quad + (1-a)^{\alpha-1}(1+b)^{\alpha-1} - (1-b)^{\alpha-1}(1+a)^{\alpha-1} = 0.
\end{align*}
Using that  $F(a,b) =\frac{(1+a)^\alpha+(1+b)^\alpha}{(1-a)^\alpha+(1-b)^\alpha} = \frac{(1+a)^\alpha+(1+b)^\alpha}{(a+b)^\alpha +2^\alpha}$,  the above equality holds if and only if 
\begin{align*}
& (a+b)^{\alpha-1} \left\{ [ (1+b)^{\alpha-1} - (1+a)^{\alpha-1}] [ (1-a)^\alpha + (1-b)^\alpha ]  \right. \\
&\qquad\ \qquad \qquad \left. + [ (1-b)^{\alpha-1} - (1-a)^{\alpha-1}] [ (1+a)^\alpha + (1+b)^\alpha ] \right\} \\
&\qquad + [ (1-a)^{\alpha-1}(1+b)^{\alpha-1} - (1+a)^{\alpha-1}(1-b)^{\alpha-1} ] [ (a+b)^\alpha + 2^\alpha] = 0,
\end{align*}
equivalently,
\begin{align*}
 2 (a+b)^{\alpha-1} \left[ (1-b^2)^{\alpha-1}\right. & \left. -  (1-a^2)^{\alpha-1} \right]  \\
& + 2^\alpha[ (1-a)^{\alpha-1}(1+b)^{\alpha-1} - (1+a)^{\alpha-1}(1-b)^{\alpha-1} ] = 0
\end{align*}
Factoring out $2(a+b)^{\alpha-1}(1-b^2)^{\alpha-1}\neq 0$ yields
\begin{equation}\label{eq:local_max_necessary} 
1-\left(\tfrac{1-b}{1-a}\right)^{1-\alpha}\left(\tfrac{1+b}{1+a}\right)^{1-\alpha} - \left(\tfrac{a+b}{2}\right)^{1-\alpha}\left[\left(\tfrac{1+b}{1+a}\right)^{1-\alpha} - \left(\tfrac{1-b}{1-a}\right)^{1-\alpha} \right] = 0
\end{equation}
Assume $a<b$. Since $a+b>0$, this implies $b>0$ and $-b <a <b$.
In particular, $a^2<b^2$. Let $x=\frac{1-b}{1-a}$ and $y=\frac{1+b}{1+a}$. Then $0<x<1<y$, and $0<xy<1$.
Note that $\frac{a+b}{2} = \frac{1-xy}{y-x}$.
We claim that the expression on the left hand side of (\ref{eq:local_max_necessary}) is negative. That is, we claim,
\[
1-(xy)^{1-\alpha} - \left(\frac{1-xy}{y-x}\right)^{1-\alpha}\left(y^{1-\alpha}-x^{1-\alpha}\right) <0.
\]
Indeed, multiplying the above inequality by $(1-xy)^{\alpha-1}>0$ yields
\[\frac{1-(xy)^{1-\alpha}}{(1-xy)^{1-\alpha}} - \frac{y^{1-\alpha} - x^{1-\alpha}}{(y-x)^{1-\alpha}} <0,\]
equivalently,
\[\frac{1-(xy)^{1-\alpha}}{(1-xy)^{1-\alpha}} - \frac{1 - (x/y)^{1-\alpha}}{(1-x/y)^{1-\alpha}} <0\]
which is valid since the function $t\mapsto \frac{1-t^{1-\alpha}}{(1-t)^{1-\alpha}}, 0<t<1$, is decreasing and $0<x/y < xy <1$.

Note that the expression on the left hand side of (\ref{eq:local_max_necessary}) is positive if $a>b$.
Thus, (\ref{eq:local_max_necessary}) holds if and only if $a=b$. Finally, notice that $F(a,a)=G(a,a)$ has unique solution $0<a<1$.
Since $F(a,a) = \left[(1+a)/(1-a)\right]^\alpha>1$, the conclusion follows. 
\end{proof}

\begin{proof}[Proof of Theorem~\ref{thm:lineflake}]

If $\alpha = 1$, the conclusion holds with $m=1$ by Proposition~\ref{prop:hyperbolic}, since the space $(\R^1,d_E)$ is $0$-hyperbolic.  Let $0< \alpha <1$. By Lemmas~\ref{lem:max_on_intersection} and~\ref{lem:max_on_diagonal} 
\begin{align*}
C_0(\R^1,d_E^\alpha) & = \max_{(t,s)\in D_0}F(t,s) = F(a,a) = \left(\tfrac{1+a}{1-a}\right)^\alpha = m^\alpha
\end{align*}
where $
m=\frac{1+a}{1-a}>1$  is the unique solution of \[2 = \left(\frac{1+a - (1-a)}{1-a}\right)^\alpha + \left(\frac{1+a + (1-a)}{1-a}\right)^\alpha = (m-1)^\alpha + (m+1)^\alpha. \] 
\end{proof}

\begin{rem} \label{}
Let $0<\alpha \leq 1$.
It is {\it not} true in general that for any metric space $(X,d)$
the inequality
$C_0(X,d^\alpha) \leq (C_0(X,d))^\alpha$ holds.
For example, if $\alpha = 1/2$ then $m= 5/4$ as in Theorem~\ref{thm:lineflake} and so
  $$C_0(\IR^1, d_E^{1/2}) = \sqrt{5}/2 >  (C_0(\IR^1, d_E))^{1/2} = \sqrt{1}=1.$$
\end{rem}


\section{Distances on Riemannian manifolds}

We show that the quadrilateral constant of the metric space associated to a Riemannian manifold of dimension greater than one is
bounded from below by $\sqrt{2}$.

\begin{prop}\label{prop:Riemannian} 
If $M$ is a Riemannian manifold of dimension greater than one and $d_M$ is the distance on $M$ induced by the given Riemannian metric then 
$C_0(M,d_M)\ge \sqrt{2}$.
\end{prop}

\begin{proof}
Let $p \in M$ and let  $\exp_p \colon T_p M \rightarrow M$ denote the Riemannian exponential map. 
The Riemannian metric on $M$ endows the tangent space, $T_p M $, with an inner product and we write $d_E$
for the corresponding Euclidean distance on $T_p M $.
For a vector $X \in T_p M$ and a scalar $t$, let $X_t = \exp_p (t X) \in M$.
If $X, Y \in  T_p M$ then
\begin{equation}\label{eq:riem}
\lim_{t \rightarrow 0}  \frac{d_M(X_t, Y_t)}{t} = d_E(X,Y).
\end{equation}
This is a consequence of the fact that  in normal coordinates $\{ x^i \}$ the components $g_{ij}(x)$ of the Riemannian
metric satisfy the estimate $|  g_{ij}(x) - \delta_{ij} | \leq C \| x \|^2$ for some $C$.

For $X, Y, Z, W \in T_p M$, not all identical,
\begin{align*}
\Delta(X_t, \, Y_t, \,  Z_t \, ,W_t) & =  \frac{d_M(X_t,Y_t)+d_M(Z_t ,W_t)}{\max\{d_M(X_t, Z_t) +d_M(Y_t, W_t), \, d_M(X_t ,W_t) +d_M(Y_t, Z_t)\}}    \\
                                                   & =  \frac{d_M(X_t,Y_t)/t+d_M(Z_t ,W_t)/t}{\max\{d_M(X_t, Z_t)/t +d_M(Y_t, W_t)/t, \, d_M(X_t ,W_t)/t +d_M(Y_t, Z_t)/t\}}
\end{align*}
By~\eqref{eq:riem}, $\lim_{t \rightarrow 0} \Delta(X_t, \, Y_t, \,  Z_t \, ,W_t)  = \Delta(X, \, Y, \,  Z, \, W)$,  where the second $\Delta$ is with respect to $d_E$.
Since $\dim M > 1$, there are orthogonal unit vectors $U, V \in  T_p M$.
Since
\[
C_0(M,d_M) \geq  \Delta(U_t, \, V_t, \,  0_t, \, (U+V)_t),
\]
it follows that 
\[
C_0(M,d_M) \geq \lim_{t \rightarrow 0} \Delta(U_t, \, V_t, \,  0_t, \, (U+V)_t) = \Delta(U, \, V, \,  0, \, U+V) = \sqrt{2},
\]
establishing the conclusion of the proposition.
\end{proof}

\begin{cor}\label{cor:nonnegsectional}
Let $M$ be a simply connected,
complete Riemannian manifold of non-positive sectional curvature
with associated distance $d_M$.
Then $C_0(M,d_M) =  \sqrt{2}$.
\end{cor}

\begin{proof}
By \cite[Chapter II.1, Theorem 1A.6]{Bridson-Haefliger},
the metric space $(M,d_M)$ is a $\CAT(0)$-space and so
$C_0(M,d_M) \leq  \sqrt{2}$ by Corollary~\ref{cor:CATzero}.
By Proposition~\ref{prop:Riemannian}, $C_0(M,d_M) \geq  \sqrt{2}$.
Thus $C_0(M,d_M) =  \sqrt{2}$.
\end{proof}

\bibliographystyle{amsalpha}

\begin{thebibliography}{00}

\bibitem{PhysRevE.89.032811}
R\'eka Albert, Bhaskar DasGupta, and Nasim Mobasheri, \emph{Topological
  implications of negative curvature for biological and social networks}, Phys.
  Rev. E \textbf{89} (2014), 032811.
  
  \bibitem{BergNik}
I.~D. Berg and I.~G. Nikolaev, \emph{Quasilinearization and curvature of
  {A}leksandrov spaces}, Geom. Dedicata \textbf{133} (2008), 195--218.
  \MR{2390077}
  
  \bibitem{Blumenthal}
Leonard~M. Blumenthal, \emph{Theory and applications of distance geometry},
  Second edition, Chelsea Publishing Co., New York, 1970. \MR{0268781}

\bibitem{Bridson1995}
Martin~R. Bridson, \emph{On the existence of flat planes in spaces of
  nonpositive curvature}, Proc. Amer. Math. Soc. \textbf{123} (1995), no.~1,
  223--235. \MR{1273477}
  
\bibitem{Bridson-Haefliger}
Martin~R. Bridson and Andr\'e Haefliger, \emph{Metric spaces of non-positive
  curvature}, Grundlehren der Mathematischen Wissenschaften [Fundamental
  Principles of Mathematical Sciences], vol. 319, Springer-Verlag, Berlin,
  1999. \MR{1744486}


\bibitem{Enflo1969}
Per Enflo, \emph{On the nonexistence of uniform homeomorphisms between
  {$L_{p}$}-spaces}, Ark. Mat. \textbf{8} (1969), 103--105. \MR{0271719}

\bibitem{Gao-Lau}
Ji~Gao and Ka-Sing Lau, \emph{On the geometry of spheres in normed linear
  spaces}, J. Austral. Math. Soc. Ser. A \textbf{48} (1990), no.~1, 101--112.
  \MR{1026841}

\bibitem{Gro:87Hyp:aa}
M.~Gromov, \emph{Hyperbolic groups}, Essays in group theory, Math. Sci. Res.
  Inst. Publ., vol.~8, Springer, New York, 1987, pp.~75--263. \MR{919829}

\bibitem{Johnson-Lindenstrauss}
William~B. Johnson and Joram Lindenstrauss, \emph{Basic concepts in the
  geometry of {B}anach spaces}, Handbook of the geometry of {B}anach spaces,
  {V}ol. {I}, North-Holland, Amsterdam, 2001, pp.~1--84. \MR{1863689}

\bibitem{JL2004}
Edmond Jonckheere and Poonsuk Lohsoonthorn, \emph{Geometry of network
  security}, Proceedings of the 2004 {A}merican {C}ontrol {C}onference, Boston,
  MA, 2004, pp.~976--981.

\bibitem{MR2837770}
Edmond Jonckheere, Poonsuk Lohsoonthorn, and Fariba Ariaei, \emph{Scaled
  {G}romov four-point condition for network graph curvature computation},
  Internet Math. \textbf{7} (2011), no.~3, 137--177. \MR{2837770}

\bibitem{Komuro-Saito}
Naoto Komuro, Kichi-Suke Saito, and Ryotaro Tanaka, \emph{On the class of
  {B}anach spaces with {J}ames constant {$\sqrt2$}}, Math. Nachr. \textbf{289}
  (2016), no.~8-9, 1005--1020. \MR{3512046}

\bibitem{LTW}
C.~J. Lennard, A.~M. Tonge, and A.~Weston, \emph{Generalized roundness and
  negative type}, Michigan Math. J. \textbf{44} (1997), no.~1, 37--45.
  \MR{1439667}

\bibitem{Nica-Spakula}
Bogdan Nica and J\'{a}n {\v{S}}pakula, \emph{Strong hyperbolicity}, Groups
  Geom. Dyn. \textbf{10} (2016), no.~3, 951--964. \MR{3551185}

\bibitem{Sato}
Takashi Sato, \emph{An alternative proof of {B}erg and {N}ikolaev's
  characterization of {$\rm CAT(0)$}-spaces via quadrilateral inequality},
  Arch. Math. (Basel) \textbf{93} (2009), no.~5, 487--490. \MR{2563595}

\bibitem{Sch:37On-:aa}
I.~J. Schoenberg, \emph{On certain metric spaces arising from {E}uclidean
  spaces by a change of metric and their imbedding in {H}ilbert space}, Ann. of
  Math. (2) \textbf{38} (1937), no.~4, 787--793. \MR{1503370}

\bibitem{MR2902244}
Yilun Shang, \emph{Lack of {G}romov-hyperbolicity in small-world networks},
  Cent. Eur. J. Math. \textbf{10} (2012), no.~3, 1152--1158. \MR{2902244}

\bibitem{TJ}
Nicole Tomczak-Jaegermann, \emph{Banach-{M}azur distances and
  finite-dimensional operator ideals}, Pitman Monographs and Surveys in Pure
  and Applied Mathematics, vol.~38, Longman Scientific \& Technical, Harlow;
  copublished in the United States with John Wiley \& Sons, Inc., New York,
  1989. \MR{993774}

\bibitem{Vai:05Gro:aa}
Jussi {V}{\"a}is{\"a}l{\"a}, \emph{Gromov hyperbolic spaces}, Expo. Math.
  \textbf{23} (2005), no.~3, 187--231. \MR{2164775}
  
\end{thebibliography}
%
\providecommand{\bysame}{\leavevmode\hbox to3em{\hrulefill}\thinspace}
\providecommand{\MR}{\relax\ifhmode\unskip\space\fi MR }
\providecommand{\MRhref}[2]{%
  \href{http://www.ams.org/mathscinet-getitem?mr=#1}{#2}
}
\providecommand{\href}[2]{#2}
%
%


\end{document}